\documentclass{amsart}

\usepackage[letterpaper,margin=1.2in]{geometry}

\usepackage{amsmath,amssymb,amsthm}
\usepackage{MnSymbol}
\usepackage{enumerate}
\usepackage{hyperref}
\usepackage{mathtools}
\usepackage{xcolor}
\usepackage{tikz}
\usepackage{tikz-cd}

\newtheorem{theorem}{Theorem}[section]
\newtheorem{observation}[theorem]{Observation}
\newtheorem{lemma}[theorem]{Lemma}
\newtheorem{proposition}[theorem]{Proposition}
\newtheorem{corollary}[theorem]{Corollary}
\newtheorem{fact}[theorem]{Fact}

\newtheorem*{theorem*}{Theorem}

\numberwithin{equation}{section}
\numberwithin{figure}{section}

\theoremstyle{definition}
\newtheorem{definition}[theorem]{Definition}
\newtheorem{notation}[theorem]{Notation}

\theoremstyle{remark}
\newtheorem{remark}[theorem]{Remark}
\newtheorem{example}[theorem]{Example}

\newcommand\R{\mathbb{R}}
\newcommand\C{\mathbb{C}}

\newcommand{\M}{\mathbb{M}}
\newcommand\N{\mathbb{N}}

\newcommand\E{\mathbb{E}}

\newcommand{\cA}{\mathcal{A}}

\newcommand{\cM}{\mathcal{M}}

\newcommand{\bX}{\mathbf{X}}

\newcommand{\bZ}{\mathbf{Z}}

\newcommand{\take}{\lefthalfcap}

\DeclareMathOperator{\id}{id}

\DeclareMathOperator{\Span}{Span}

\DeclareMathOperator{\re}{Re}
\DeclareMathOperator{\im}{Im}
\DeclareMathOperator{\Tr}{Tr}
\DeclareMathOperator{\tr}{tr}

\DeclareMathOperator{\sa}{sa}
\DeclareMathOperator{\ev}{ev}
\DeclareMathOperator{\Perm}{Perm}
\DeclareMathOperator{\supp}{supp}

\DeclareMathOperator{\Div}{div}

\DeclareMathOperator{\Cyc}{Cyc}

\DeclareMathOperator{\Trans}{Trans}
\DeclareMathOperator{\crss}{cr}

\DeclareMathOperator{\sgn}{sgn}

\DeclarePairedDelimiter{\ip}{\langle}{\rangle}

\title{Combinatorial aspects of Parraud's asymptotic expansion for GUE matrices}
\author{David Jekel}
\address{Department of Mathematical Sciences, University of Copenhagen, \newline Universitetsparken 5, 2100 Copenhagen {\O}, Denmark}
\email{daj@math.ku.dk}

\begin{document}
	
	\maketitle
	
	\begin{abstract}
		We give a new combinatorial proof of Parraud's formula for the asymptotic expansion in powers of $1/N^2$ for the expected trace of polynomials of several independent $N \times N$ GUE matrices, which expresses the result using a mixture of free difference quotients, introducing new freely independent semicircular variables, and integration with respect to parameters.  Our approach streamlines the statement of the formula while clarifying its relationship to the combinatorial genus expansion.  
	\end{abstract}
	
	\section{Introduction}
	
	\subsection{Motivation}
	
	If $p$ is a non-commutative polynomial in $m$ variables and $X_1^{(N)}$, \dots, $X_m^{(N)}$ are independent matrices drawn from the Gaussian unitary ensemble (GUE), then there is an asymptotic expansion of $\E [\tr_N[p(X_1^{(N)},\dots,X_m^{(N)})]]$ in powers of $1/N^2$ (see e.g.\ \cite[\S 1.7]{MingoSpeicher}).  When $p$ is a monomial, the coefficients are given by the enumeration of maps of a certain genus.  In physics, the connection between matrix moments and planar maps was studied by 't Hooft in 1974 \cite{tHooft1974A,tHooft1974B}, and then generalized to arbitrary genus by Br{\'e}zin, Itzykson, Parisi, and Zuber in 1978 \cite{BIPZ1978}.  In 1986, Harer and Zagier used this machinery to study the Euler characteristics of moduli spaces of curves \cite{HZ1986}.  The applications to random matrix theory were further developed in \cite{EM2003partition,GMS2006,Shcherbina2014}.  See \cite{Zvonkin1997} for survey of map enumeration and random matrices.
	
	However, due its purely combinatorial nature, the topological expansion is difficult to apply to non-commutative smooth functions beyond the setting of polynomials and power series.  Thus it is more useful for analytic applications to have an expression for the asymptotic expansion in terms of non-commutative derivatives of the input function, which allows extension of the formula to smooth functions as in the work of Schultz \cite{Schultz2005}.  For instance, to prove sharp results about the convergence of operator norms (equivalently, the convergence of their spectra in Hausdorff distance), one would like an asymptotic expansion for $\E [\tr_N[f(p(X_1^{(N)}, \dots, X_m^{(N)}))]]$ where $f$ is a smooth function and $p$ is a non-commutative polynomial.  Parraud gave an alternative approach to asymptotic expansions for $\E[\tr_N[f(X_1^{(N)},\dots,X_m^{(N)})]]$ for certain functions $f$ where the terms and error bounds in the expansion are expressed using non-commutative derivatives of $f$ \cite[Theorem 1.1, Theorem 3.4]{Parraud2023}.
	
	This allowed for analytic control over the terms and the error bounds using tools such operator derivatives and Fourier decomposition of smooth functions (see \cite{Nikitopoulos2023}), and hence to obtain new proofs and generalizations of Haagerup and Thorbj{\o}rnsen's result on convergence of operator norms of functions of GUE matrices \cite{HaagThorbNormBound}.  For instance, Collins, Guionnet, and Parraud used only the first-order case of the asymptotic expansions to get sharper estimates for operator norms of matrix polynomials in GUE matrices in \cite{CGP2022}.  Moreover, Belinschi and Capitaine \cite{BelCap2022} used Parraud's asymptotic expansion formula to study the operator norms of functions of tensor products of GUE matrices, and their main theorem implies, by the work of Hayes \cite{HayesPT}, the Peterson--Thom conjecture about the free group von Neumann algebra from \cite[p.~590]{PetersonThom}.  Parraud recently gave another proof of strong convergence for tensors using the asymptotic expansions of smooth functions \cite{Parraud2024tensor}.  Moreover, a new approach to strong convergence for tensors of Chen, Garza-Vargas, and Van Handel \cite{CGVVH2024} uses asymptotic expansions together with a classical complex-analytic estimate that had hitherto not been used in this context.  Analogous results were also proved for Haar unitaries:  Parraud gave an asymptotic expansion in \cite{Parraud2023unitary}.  Bordenave and Collins studied the strong convergence of tensor products for Haar unitaries in \cite{BorCol2023}, but their proof does not use Parraud's formula.
	
	Given the significant applications of asymptotic expansions for GUE matrices, this paper aims for a better understanding of Parraud's asymptotic expansion formula and how it relates to the genus expansion.  Of course, the terms in Parraud's asymptotic expansion must agree with the terms in the genus expansion by uniqueness; see \cite[Remark 3.8]{Parraud2023} and \cite[Corollary 1.3]{PS2023energy}.  However, since Parraud's proof is analytic, it does not provide any combinatorial intuition for why these two things should be the same.  We will give a new combinatorial proof of Parraud's formula that shows how the derivative operators in the formula relate to crossings in the construction of maps from pair partitions, and in turn to the genus.
	
	At the same time, we aim to streamline the statement of the formula.  Though Parraud's asymptotic expansion gives fine analytic control, it is quite challenging to keep track of the many terms.  Indeed, for each higher order term, one has to increase the number of variables and plug in additional freely independent semicircular variables, indexed by complicated sets defined inductively \cite[\S 2.3]{Parraud2023}.  Then one must plug in a linear combination of the semicircular variables depending on some parameters and integrate over the parameters.  Here we will lay out methodically the differential and algebraic operations used in the formula, and define crossing derivative operators $D^{\crss}$ and $T^{\crss}$, which allow for a more compact statement of the expansion up to arbitrary order.
	
	Here we focus only on the case of polynomials in order to keep the paper self-contained and manageable in length.  Of course, the smooth setting is important for applications, and we plan to address this setting in later joint projects.  Briefly speaking, we believe the right setting for Parraud's formula should be some space of non-commutative $C^k$ functions by taking the completion of the space of non-commutative polynomials with respect to certain non-commutative $C^k$ norms, which would unify the \emph{ad hoc} combination of polynomials, complex exponentials, and smooth univariate functions in Parraud's work.  Then Theorem \ref{thm:Parraud1} would extend from polynomials to these more general functions by a density argument, after showing that the derivatives and algebraic operations used in the formula extend to non-commutative smooth functions.  These spaces of non-commutative functions would be handled in a similar way to \cite{DGS2021} and \cite[\S 3]{JLS2022}, although the exact function spaces in those papers are not well-suited to Parraud's formula.
	
	Another line of investigation that we leave open for future research is to extend the combinatorial proof of Parraud's formula to other ensembles, such as GOE, GSE, Haar unitary, Haar orthogonal.  In the GOE case, for instance, the combinatorial formulas are similar but the potential for twists in the construction of maps results in non-orientable surfaces contributing to the genus expansion and would require adding another type of derivative into our formula.  In the Haar unitary and orthogonal matrix cases would be studied using the Weingarten calculus.

	
	The paper is organized as follows:
	\begin{itemize}
		\item \S \ref{subsec: statement} describes the statement of the first-order Parraud's formula in the case of non-commutative polynomials, and \S \ref{subsec: higher order} describes the higher-order formula.
		\item In \S \ref{sec: preliminaries}, we give background on GUE matrices, free independence, and semicircular variables, and in particular integration by parts needed for the proofs.
		\item In \S \ref{sec: genus expansion}, we give a generalized genus expansion which handles a mixture of GUE matrices and freely independent semicircular variables.
		\item In \S \ref{sec: Parraud formula proof}, we give the combinatorial proof of Theorem \ref{thm:Parraud1}.
	\end{itemize}

	\subsection{Statement of Parraud's formula} \label{subsec: statement}
	
	In order to describe certain algebraic operations, it will be convenient to use vector spaces rather than sets of indices to keep track of the variables.  Although the added abstraction of the vector space notation may be off-putting, it will enable us to compactly write the index sets for the iterated Parraud's formula using direct sums and tensor product operations, so we prefer this over the involved construction of index sets as subsets of the natural numbers in \cite[Definition 2.17]{Parraud2023}.  The vector space formalism is partly motivated by Voiculescu's free Gaussian functor (see \cite[\S 1.5]{ShlyakhtenkoParkCity}).  Recall that if $V$ is a real inner-product space, then Voiculescu's free Gaussian constructions give a tracial von Neumann algebra $\cM_V$ and self-adjoint operators $X_V = (X_v)_{v \in V}$ in $\cM_V$, such that
	\begin{itemize}
		\item The mapping $v \mapsto X_v$ is $\R$-linear.
		\item For any orthonormal family $e_1$, \dots, $e_k \in V$, the operators $X_{e_1}$, \dots, $X_{e_k}$ are freely independent standard semicircular operators.
	\end{itemize}
	There is also a random matrix analog of this construction.  Namely, for each real inner-product space $V$ and $N \in \N$, there exist random $N \times N$ self-adjoint matrices $X_V^{(N)} = (X_v^{(N)})_{v \in V}$ such that
	\begin{itemize}
		\item The mapping $v \mapsto X_v^{(N)}$ is $\R$-linear.
		\item For any orthonormal family $e_1$, \dots, $e_k \in V$, the operators $X_{e_1}^{(N)}$, \dots, $X_{e_k}^{(N)}$ are independent standard GUE matrices.
	\end{itemize}
	To show existence, one can complete $V$ to a real Hilbert space, choose $X_v^{(N)}$ for $v$ ranging over an orthonormal basis for $V$, and then extend by linearity.
	
	As our random variables are indexed by a vector space, we also consider algebras of $*$-polynomials in variables indexed by a vector space.  Thus, for a complex $*$-algebra $\cA$ and a real vector space $V$, define $\cA\ip{V}$ as the free unital $*$-algebra generated by a copy of $\cA$ and formal self-adjoint variables $(x_v)_{v \in V}$ modulo the relations $x_{v_1+v_2} = x_{v_1} + x_{v_2}$ and $x_{\lambda v} = \lambda x_v$ for $\lambda \in \R$.  Equivalently, $\cA\ip{V}$ is the unital $*$-algebraic free product of $\cA$ and the tensor algebra of $V$.  Below we will take $\cA = \M_N$, the algebra of $N \times N$ matrices over $\C$.
	
	Fix real inner-product spaces $V$ and $W$.  Let $\cM_W$ be the free Gaussian functor of $W$, and let $\M_N * \cM_W$ be the free product as tracial von Neumann algebras.  View $\M_N$ and $\cM_W$ as unital subalgebras of $\M_N* \cM_W$ and hence view the Gaussian matrices $X_v^{(N)}$ as random elements of $\M_N * \cM_W$, and similarly the semicirculars $X_w$ as elements of $\M_N * \cM_W$.  Note that there is a unique $*$-homomorphism $\ev_{N,V,W}: \M_N\ip{V \oplus W} \to \M_N * \cM_W$ (which we will call the \emph{evaluation map}) that sends the formal variable $x_{v \oplus w}$ to $X_v^{(N)} + X_w$ (note the output is random).  We will write $f(X_V^{(N)},X_W) = \ev_{N,V,W}(f)$ for $f \in \M_N\ip{V \oplus W}$.
	
	Parraud's formula gives a way to describe the difference between the expected trace of $f(X_V^{(N)},X_W)$ and $\tr_{\cM_{V \oplus W}}(f(X_V,X_W))$, where $X_V$ is a free semicircular family and we view $X_V$ and $X_W$ as elements of $\cM_{V \oplus W} \subseteq \M_N * \cM_{V \oplus W}$.  The first iteration of the formula can be stated as follows.
	
	\begin{theorem} \label{thm:Parraud1}
		With the notation above, we have
		\begin{multline*}
			\E \circ \tr_{\M_N * \cM_V}[f(X_V^{(N)},X_W)] = \tr_{\cM_{V \oplus W}}[f(X_V,X_W)] \\
			+ \frac{1}{N^2} \E \circ \tr_{\M_N * \cM_{(V \oplus W) \otimes \R^6}}[T_{V,W}^{\crss} f(X_V^{(N)},X_{(V \oplus W) \otimes \R^6})],
		\end{multline*}
		where $T_{V,W}^{\crss}$ is the linear operator $\M_N\ip{V \oplus W} \to \M_N\ip{V \oplus [(V \oplus W) \otimes \R^6]}$ described by Definition \ref{def: T cross} below.  Here the tensor products of the indexing vector spaces are taken over $\R$.
	\end{theorem}
	
	The utility of the formula comes from the fact that $T_{V,W}^{\crss}$ is an explicit combination of non-commutative derivatives and mappings of polynomial algebras induced by linear substitutions of the input variables, as described below.  The operator $T_{V,W}^{\crss}$ depends on $N$ since in fact its input $f$ is a non-commutative polynomial with coefficients in $\M_N$, but we suppress the dependence in the notation as the role of deterministic matrices in the formulas is the same for all $N$.  We now describe the construction of $T_{V,W}^{\crss}$, and we proceed more generally replacing $\M_N$ with an arbitrary $*$-algebra $\cA$.  That is, we will define $T_{V,W}^{\crss}$ for a general coefficient algebra $\cA$, and for Theorem \ref{thm:Parraud1} use the special case $\cA = \M_N$.  The definition of $T_{V,W}^{\crss}$ will be built up in stages, and the dependence on $\cA$ and $V$ and $W$ is suppressed in some of the notations as these objects are fixed throughout the whole discussion.
	
	\begin{observation}[Linear maps induce maps of polynomial algebras] \label{obs: induced map on polynomials}
		Given $\cA$ and real vector spaces $V$ and $W$ and a linear map $\phi: V \to W$, there is a unique $*$-algebra homomorphism $\phi_* = \cA\ip{\phi}: \cA\ip{V} \to \cA\ip{W}$ that restricts to the identity on $\cA$ and maps $x_v$ to $x_{\phi(v)}$ for all $v \in V$.  Thus, the mapping $V \mapsto \cA\ip{V}$ is a functor from real vector spaces to complex $*$-algebras.
	\end{observation}
	
	\begin{notation}[Free difference quotient, cf. {\cite[\S 3]{VoiculescuFE5}}] \label{not: free diff quot}
		Fix a $*$-algebra $\cA$ and real vector spaces $V$ and $W$.  Let $V_{\C}$ denote the complexification of $V$. We define
		\[
		\partial_V: \cA\ip{V \oplus W} \to V_{\C} \otimes_{\C} \cA\ip{V \oplus W} \otimes_{\C} \cA\ip{V \oplus W}
		\]
		as the unique derivation such that
		\begin{align*}
			\partial_V|_{\cA} &= 0 \\
			\partial_V(x_v) &= v \otimes 1 \otimes 1 \text{ for } v \in V \\
			\partial_V(x_w) &= 0 \text{ for } w \in W.
		\end{align*}
		When we say that $\partial_V$ is a derivation, we mean that
		\[
		\partial_V(f(x)g(x)) = f(x) \partial g(x) + \partial f(x) g(x),
		\]
		where the first term on the right-hand side uses the multiplication map sending $(a, v \otimes b \otimes c)$ to $v \otimes ab \otimes c$ and the second term uses the multiplication map sending $(v \otimes a \otimes b, c)$ to $v \otimes a \otimes bc$.
	\end{notation}
	
	\begin{remark} \label{rem: canonical basis index}
	The role of the left-most tensor $V$ is to keep track of ``which variable'' has been differentiated.  Indeed, consider the case where $V = \R^m$, $W = \R^n$, and $V \oplus W$ is identified with $\R^{m+n}$.  Let $x_j$ be the variable associated to the standard basis vector $e_j$.  By linearity, the variables $(x_{v\oplus w})_{v \oplus w \in V \oplus W}$ can be represented by $(x_1,\dots,x_{m+n})$.  Let $\partial_j$ be the free difference quotient with respect to the variable $x_j$ from \cite[\S 3]{VoiculescuFE5}.  Then
	\[
	\partial_V f(x_V,x_W) = \sum_{j=1}^m e_j \otimes \partial_j f(x_1,\dots,x_{m+n}).
	\]
	In other words, $\partial_V$ can be understood as the $m$-tuple of free difference quotients $\partial_1 f$, \dots, $\partial_m f$ in each variable.
	\end{remark}
	
	The uniqueness of the free difference quotient map is immediate from the fact that it is a derivation and $\cA \ip{V}$ is generated by $\cA$ and the $x_v$'s.  To prove existence, one can define the map on monomials of each degree explicitly (see \cite[\S 8.1]{MingoSpeicher}), and check that it satisfies the desired properties.  Since we are working in the vector space framework, one also uses multilinearity to verify the well-definedness, e.g. on for the monomials of degree one, since $(v \oplus w) \mapsto v \otimes 1 \otimes 1$ is linear, and so therefore, there is a well-defined map from $\Span(x_u: u \in V \oplus W) \to V_{\C} \otimes_{\C} \cA\ip{V \oplus W} \otimes_{\C} \cA\ip{V \oplus W}$.
	
	\begin{remark}[Real versus complex tensor products]
	We use real vector spaces $V$ to index the variables, but at the same time complex $*$-algebras are natural in the operator-algebraic setting.  Thus, both real and complex tensor products appear in this paper.  In Notation \ref{not: free diff quot}, we took the complexification $V_{\C}$ in order to avoid mixing real and complex tensor products in the same equation.  Generally, real and complex vector spaces in this paper can be distinguished based on context:  The indexing vector spaces are real, but the spaces of polynomials are complex.  An expression where $V$ appears is a real vector space with tensor products occurring over $\R$, and an expression where $V_{\C}$ appears is a complex vector space with tensor products occurring over $\C$.  Thus, we will suppress the subscripts $\R$ and $\C$ on the tensor signs in the rest of the paper.  The distinction between $\R$ and $\C$ plays little role in the arguments in any case.
	\end{remark}
	
	Now building up to the definition of crossing derivatives, we need several more ingredients, including the cyclic gradient and the free Laplacian.
	
	\begin{notation}[Permuted multiplication maps]
		For any algebra $\cA$, let $\mathfrak{m}: \cA \otimes \cA \to \cA$ be the multiplication map.  More generally, given $m \in \N$ and a permutation $[j_1 \dots j_m]$ of $[m]$ (here the permutation is written as a list, not in cycle notation), we write
		\[
		\mathfrak{m}_{j_1,\dots,j_m}: \cA^{\otimes m} \to \cA: a_1 \otimes \dots \otimes a_m \mapsto a_{j_1} \dots a_{j_m}.
		\]
	\end{notation}
	
	\begin{notation}[Cyclic gradient] \label{not: cyclic derivative}
		Given $\cA$, $V$, and $W$ as above, the \emph{cyclic gradient} $D_V^\circ$ is the map
		\[
		D_V^\circ = (\id_V \otimes \mathfrak{m}_{2,1}) \circ \partial_V: \cA\ip{V \oplus W} \to V \otimes \cA\ip{V \oplus W}.
		\]
	\end{notation}
	
	\begin{notation}[Pairing map]
		For a real inner-product space $V$, we denote by $\mathfrak{n}_V: V \otimes V \to \R$ the map $v \otimes w \mapsto \ip{v,w}$.  We also denote by $\mathfrak{n}_V$ the $\C$-bilinear (not sesquilinear) extension to the complexification $V_{\C} \otimes_{\C} V_{\C} \to \C$.
	\end{notation}
	
	\begin{definition}
		Given $\cA$, $V$, and $W$ as above, the (algebraic) \emph{free Laplacian} $L_V: \cA\ip{V \oplus W} \to \cA\ip{V \oplus W} \otimes_{\C} \cA\ip{V \oplus W}$ is given by the following composition of maps:
		\[
		\begin{tikzcd}
			\cA\ip{V \oplus W} \arrow{d}{D_V^\circ} \\
			V_{\C} \otimes \cA\ip{V \oplus W} \arrow{d}{\id_V \otimes \partial_V} \\
			V_{\C} \otimes V_{\C} \otimes \cA\ip{V \oplus W} \otimes \cA\ip{V \oplus W} \arrow{d}{\mathfrak{n}_V \otimes \id_{\cA\ip{V \oplus W}} \otimes \id_{\cA\ip{V \oplus W}}} \\
			\cA\ip{V \oplus W} \otimes \cA\ip{V \oplus W}
		\end{tikzcd}
		\]
		We remark that if $V = \R^m$ and $W = \R^n$ and $x_j = x_{e_j}$, we have
		\[
		L_V f = \sum_{j=1}^m \partial_j \circ D_j^\circ f.
		\]
	\end{definition}
	
	\begin{definition} \label{def: crossing derivative}
		Given real inner-product spaces $V$ and $W$, we define the \emph{crossing derivative} $D_{V,V}^{\crss}: \cA\ip{V \oplus W} \to \cA\ip{(V \oplus W) \otimes \R^4}$ as follows.  Let $\id_{V \oplus W} \otimes e_j: V \oplus W \to (V \oplus W) \otimes \R^2$ denote the map $(v \oplus w) \mapsto (v \oplus w) \otimes e_j$ for $j = 1, \dots, 4$.  Then $D_{V,V}^{\crss}$ is the operator given by the following composition of maps:
		\[
		\begin{tikzcd}
			\cA\ip{V \oplus W}
			\arrow{d}{\frac{1}{2} L_V} \\
			\cA\ip{V \oplus W} \otimes \cA\ip{V \oplus W}
			\arrow{d}{\partial_{V} \otimes \partial_{V}} \\
			\bigl(V_{\C} \otimes \cA\ip{V \oplus W} \otimes \cA\ip{V \oplus W} \bigr) \otimes \bigl(V_{\C} \otimes \cA\ip{V \oplus W} \otimes \cA\ip{V \oplus W} \bigr)
			\arrow{d}{\cong} \\
			\bigl(V_{\C} \otimes V_{\C}\bigr) \otimes \cA\ip{V \oplus W} \otimes \cA\ip{V \oplus W} \otimes \cA\ip{V \oplus W} \otimes \cA\ip{V \oplus W}
			\arrow{d}{\mathfrak{n}_{V} \otimes (\id_{V \oplus W} \otimes e_1)_* \otimes (\id_{V \oplus W} \otimes e_2)_* \otimes (\id_{V \oplus W} \otimes e_3)_* \otimes (\id_{V \oplus W} \otimes e_4)_*} \\
			\cA\ip{(V \oplus W) \otimes \R^4} \otimes \cA\ip{(V \oplus W) \otimes \R^4} \otimes \cA\ip{(V \oplus W) \otimes \R^4} \otimes \cA\ip{(V \oplus W) \otimes \R^4}
			\arrow{d}{\mathfrak{m}_{4,3,2,1}} \\
			\cA\ip{(V \oplus W) \otimes \R^4}.
		\end{tikzcd}
		\]
		The third map is the isomorphism that swaps the order of the tensor product while keeping the two copies of $V$ in the same order and the four copies of $\cA\ip{(V \oplus W) \otimes \R^4}$ in the same order.
		
		We also define $D_{V,W}^{\crss}$ similarly with the first map still being $L_V$, the second being $\partial_W \otimes \partial_W$, and the fourth map having $\mathfrak{n}_W$ instead of $\mathfrak{n}_V$.
	\end{definition}
	
	The crossing derivatives map $D_{V,V}^{\crss}$ and $D_{V,W}^{\crss}$ map from $\cA\ip{V \oplus W}$ to $\cA\ip{(V \oplus W) \otimes \R^4}$, which means $D_{V,V}^{\crss}f$ and $D_{V,W}^{\crss}f$ are functions of ``four times as many variables'' as $f$ is.  The application of third operation in the composition creates four copies of each variable in order to keep track of which of the four tensorands they were in, before they are later mixed together by multiplication at the final step.  Moreover, the crossing derivative is a ``fourth order free differential operator'' overall since it has an application of $L_V$ and then two applications of $\partial_{V}$ or $\partial_W$, one on each side of the tensor sign.  We call these operators ``crossing derivatives'' because of their relationship with crossings for partitions in the genus expansion.  Heuristically, for a monomial $f(x) = x_{u(1)} \dots x_{u(k)}$ where each $u(j)$ is either in $V$ or $W$, the $L_V f$ is a linear combination of terms where two of the $x_{u(j)}$'s for $u(j) \in V$ have been removed.  Then applying $\partial_V \otimes \partial_V$ creates terms where another pair of $x_{u(j)}$'s has been removed which forms a crossing with the first pair.
	
	\begin{example}[Crossing derivative of a sixth-degree monomial]
	Let us illustrate the definition $D_{V,V}^{\crss}$ with an example.  Here for simplicity we take $W = 0$, so only $V$ appears. Let $f$ be a sixth-degree mononomial
	\[
	f(x) = x_{v(1)} x_{v(2)} x_{v(3)} x_{v(4)} x_{v(5)} x_{v(6)}.
	\]
	Then
	\begin{multline*}
	L_V = \sum_{1 \leq i < j \leq 6} \ip{v(i),v(j)} \Bigl(x_{v(i+1)} \dots x_{v(j)} \otimes x_{v(j+1)} \dots x_{v(6)} x_{v(1)} \dots x_{v(i-1)} \\
	+ x_{v(j+1)} \dots x_{v(6)} x_{v(1)} \dots x_{v(i-1)} \otimes x_{v(i+1)} \dots x_{v(j)}
	\Bigr).
	\end{multline*}
	There are $\binom{6}{2} \times 2 = 30$ terms in all.  For example, the terms with $i = 2 < j$ are
	\begin{align*}
	&\ip{v(2),v(3)} \Bigl( 1 \otimes x_{v(4)} x_{v(5)} x_{v(6)} x_{v(1)} + x_{v(4)} x_{v(5)} x_{v(6)} x_{v(1)} \otimes 1 \Bigr) \\
	&\ip{v(2),v(4)} \Bigl( x_{v(3)} \otimes x_{v(5)} x_{v(6)} x_{v(1)} + x_{v(5)} x_{v(6)} x_{v(1)} \otimes x_{v(3)} \Bigr) \\
	&\ip{v(2),v(5)} \Bigl( x_{v(3)} x_{v(4)} \otimes x_{v(6)} x_{v(1)} + x_{v(6)} x_{v(1)} \otimes x_{v(3)} x_{v(4)} \Bigr) \\
	&\ip{v(2),v(6)} \Bigl( x_{v(3)} x_{v(4)} x_{v(5)} \otimes x_{v(1)} + x_{v(1)} \otimes x_{v(3)} x_{v(4)} x_{v(5)} \Bigr).
	\end{align*}
	The factor of $(1/2)$ in front of $L_V$ serves to counteract the double counting that happens due to the terms occuring in flip-symmetric pairs.  After applying $\tfrac{1}{2} L_V$, we apply $\partial_V \otimes \partial_V$.  If there is $1$ on one side of the tensor sign, then the corresponding term will vanish since $\partial_V(1) = 0$.  This occurs when the indices $i$ and $j$ are cyclically consecutive.  For instance, the terms $1 \otimes x_{v(4)} x_{v(5)} x_{v(6)} x_{v(1)} + x_{v(4)} x_{v(5)} x_{v(6)} x_{v(1)} \otimes 1 $ from indices $i = 2$, $j = 3$ will vanish under $\partial_V \otimes \partial_V$.  Thus, only $20$ of the terms from $L_V$ will survive application of $\partial_V \otimes \partial_V$.  In general, when we apply $\partial_V \otimes \partial_V$ to each of the terms from $L_V f$, the first copy of $\partial_V$ will differentiate every possible variable on the left side of the tensor sign and the second $\partial_V$ will differentiate every possible variable on the right side of the tensor sign.  For instance, taking the term with $i = 2$, $j = 4$ from above, we compute
	\begin{align*}
	& \partial_V \otimes \partial_V \left[ x_{v(3)} \otimes x_{v(5)} x_{v(6)} x_{v(1)} \right] \\
	=& \bigl( v(3) \otimes 1 \otimes 1 \bigr) \otimes \bigl( v(5) \otimes 1 \otimes  x_{v(6)} x_{v(1)} \bigr) \\
	&+ \bigl( v(3) \otimes 1 \otimes 1 \bigr) \otimes \bigl( v(6) \otimes x_{v(5)} \otimes x_{v(1)} \bigr) \\
	&+ \bigl( v(3) \otimes 1 \otimes 1 \bigr) \otimes \bigl( v(1) \otimes x_{v(5)} x_{v(6)} \otimes 1 \bigr);
	\end{align*}
	here for ease of notation we dropped the factors $(1/2)$ and $\ip{v(2),v(4)}$ which will need to be included in the final expression.
	The next operation in our definition is simply to reorder the tensorands to put the vectors at the beginning, for notational convenience.  Thus, for instance,
	\[
	\bigl( v(3) \otimes 1 \otimes 1 \bigr) \otimes \bigl( v(6) \otimes x_{v(5)} \otimes x_{v(1)} \bigr) \mapsto \bigl( v(3) \otimes v(6) \bigr) \otimes 1 \otimes 1 \otimes x_{v(5)} \otimes x_{v(1)}.
	\]
	We next apply the bilinear pairing $\mathfrak{n}_V$ to the vectors as well as transform the variables over $V$ into variables over $V \otimes \R^4$.  Thus, for instance,
	\[
	\bigl( v(3) \otimes v(6) \bigr) \otimes 1 \otimes 1 \otimes x_{v(5)} \otimes x_{v(1)} \mapsto \ip{v(3),v(6)} \otimes 1 \otimes 1 \otimes x_{v(5) \otimes e_3} \otimes x_{v(1) \otimes e_4}.
	\]
	Finally, we apply the permuted multiplication $\mathfrak{m}_{4,3,2,1}$ to get
	\[
	\ip{v(3),v(6)} \otimes 1 \otimes 1 \otimes x_{v(5) \otimes e_3} \otimes x_{v(1) \otimes e_4} \mapsto \ip{v(3),v(6)} x_{v(1) \otimes e_4} x_{v(5) \otimes e_3}.
	\]
	Including the factors that we dropped before, we see that one term in $D_{V,V}^{\crss} f$ will be
	\[
	\tfrac{1}{2} \ip{v(2),v(4)} \ip{v(3),v(6)} x_{v(1) \otimes e_4} x_{v(5) \otimes e_3}.
	\]
	The index pairs $\{2,4\}$ and $\{3,6\}$ represent a crossing, i.e., a pair of index pairs $\{i,j\}$ and $\{i',j'\}$ such that $i < i' < j < j'$, or more generally the two pairs are interspersed in an alternating fashion when the four indices are written in increasing order.  To give another example, suppose that in $L_V$ we looked at the second term with $i = 3$ and $j = 6$, namely,
	\[
	\ip{v(3),v(6)} x_{v(1)} x_{v(2)} \otimes x_{v(4)} x_{v(5)},
	\]
	and then suppose that in the application of $\partial_V \otimes \partial_V$ we choose the term which differentiates $x_{v(1)}$ on the left and $x_{v(5)}$ on the right.  Then the final contribution to the crossing derivative would be
	\[
	\tfrac{1}{2} \ip{v(3),v(6)} \ip{v(1),v(5)} x_{v(4) \otimes e_3} x_{v(2) \otimes e_2}.
	\]
	Overall, $D_{V,V}^{\crss} f$ has terms arising from each choice of indices forming a crossing.  Every set of $4$ elements from $\{1,\dots,6\}$ can be uniquely partitioned into two pairs that form a crossing, resulting in $\binom{6}{4} = 15$ possibilities.  But for each crossing, there is a choice of which pair came from $L_V$ and which pair came from $\partial_V \otimes \partial_V$, and additionally which term is on which side of the tensor side in $L_V$; both of these choices affect the final assignment of vectors in $V \otimes \R^4$ and the final order in which the variables will be multiplied.  Hence, the total number of terms is $15 \cdot 4 = 60$.
	\end{example}
	
	Finally, we can define the operator $T_{V,W}^{\crss}$ occurring in Parraud's formula.
	
	\begin{definition} \label{def: T cross}
		Let $\cA$ be a $*$-algebra and $V$, $W$ real inner-product spaces.  For $s, t \in [0,1]$, define the linear map $\alpha_{s,t}$ by
		\begin{align}
			\alpha_{s,t}: (V \oplus W) \otimes \R^4 &\to V \oplus [(V \oplus W) \otimes \R^6] \nonumber \\
			(v \oplus w) \otimes e_j &\mapsto t^{1/2}v \oplus [((1-t)^{1/2} v \oplus w) \otimes (s^{1/2} e_{5+\delta(j,2)+\delta(j,3)} + (1-s)^{1/2} e_j)], \label{eq:alphast}
		\end{align}
		where $\delta(j,k)$ is the Kronecker $\delta$.  Let $(\alpha_{s,t})_*: \C \ip{(V \oplus W) \otimes \R^4} \to \C \ip{V \oplus [(V \oplus W) \otimes \R^6]}$ be the induced map on polynomial algebras per Observation \ref{obs: induced map on polynomials}. We then define the operator
		\[
		T_{V,W}^{\crss}: \cA\ip{V \oplus W} \to \cA\ip{V \oplus [(V \oplus W) \otimes \R^6]}
		\]
		by
		\[
		T_{V,W}^{\crss}(f) := \int_0^1 \int_0^1 (\alpha_{s,t})_* \circ [(1 - t)D_{V,V}^{\crss} + D_{V,W}^{\crss}](f)\,ds\,dt.
		\]
	\end{definition}
	
	The integral may be understood purely algebraically in the following sense:  For each $f$, the degree of $D_{V,V}^{\crss}(f)$ is less than or equal to $\deg(f) - 4$.  Because $\cA \ip{\alpha_{s,t}}$ does not increase the degree, $(\alpha_{s,t})_* [(1 -t)D_{V,V}^{\crss} + D_{V,W}^{\crss}] f$ is an element of the finite-dimensional vector space $U$ consisting of polynomials of degree at most $\deg(f)$ with coefficients drawn from a certain finite-dimensional subspace of $\cA$. Of course, the coefficients are continuous (even polynomial) functions of $(s,t)$.  We can thus fix a basis for $U$ and integrate each coefficient in this basis individually; one can check that the result does not depend on the choice of basis.
	
	Finally, we remark on differences between our statement and Parraud's.  We have used $t^{1/2}$ and $(1 - t)^{1/2}$ for $t \in [0,1]$ in roughly the place where Parraud uses $e^{-t/2}$ and $(1 - e^{-t})^{1/2}$ for $t \in [0,\infty)$.  This is more convenient for the combinatorial proof, and because it makes the integration occur over a finite interval.  Furthermore, since we have restricted our attention to polynomials, there are no exponential terms in our formula.  If the formula were extended to smooth functions and applied to $e^{i p(x)}$, then one would have to use the identity
	\[
	\partial [e^{ip(x)}] = \int_0^1 e^{i(1-u)p(x)} \partial p(x) e^{iup(x)}\,du
	\]
	in conjunction with the formula in Theorem \ref{thm:Parraud1} above, resulting in integration over even more parameters.

	\subsection{Iterated Parraud's formula} \label{subsec: higher order}
	
	By applying Theorem \ref{thm:Parraud1} iteratively, one obtains the following result.
	
	\begin{corollary} \label{cor: iterated Parraud}
		Let $V$ and $W$ be real inner product spaces.  Define inductively $W^{(0)} = W$ and $W^{(k+1)} = (V \oplus W^{(k)}) \otimes \R^6$.  Then for $m \in \N$,
		\begin{multline*}
			\E \circ \tr_{\M_N * \cM_V}[f(X_V^{(N)},X_W)] = \sum_{k=0}^m \frac{1}{N^{2k}} \tr_{\cM_{V \oplus W^{(k)}}}[T_{V,W^{(k-1)}}^{\crss} \dots T_{V,W^{(0)}}^{\crss,(N)} f(X_V,X_{W^{(k)}})] \\
			+ \frac{1}{N^{2(m+1)}} \E \circ \tr_{\cM_{W^{(m+1)}}}[T_{V,W^{(m)}}^{\crss} \dots  T_{V,W^{(0)}}^{\crss} f(X_V^{(N)},X_{W^{(m+1)}})].
		\end{multline*}
	\end{corollary}
	
	As further motivation, we observe a rough parallel between this asymptotic expansion and Taylor's formula with integral remainder, which says that for a $C^{m+1}$ function $f$,
	\[
	f(1) = \sum_{k=0}^m \frac{1}{k!} f^{(k)}(0) + \frac{1}{m!} \int_0^1 (1 - t)^m f^{(m+1)}(t)\,dt.
	\]
	In the version of Parraud's formula stated here, both the derivative terms and the remainder term involve integration over $[0,1]$ of some powers of $t_j$, $1 - t_j$, $s_j$, and $1 - s_j$, where $(s_j,t_j)$ are the parameters introduced in the $j$th occurrence of an operator $T^{\crss}$.  The proof is also based on interpolating between $t = 0$ and $t = 1$, expressing $f(1) - f(0)$ as an integral, and then repeating this process for higher order terms.
	
	The parallel is more apparent if we use Fubini's theorem rather than integration by parts in Taylor's formula, which results in
	\[
	f(1) = \sum_{k=0}^m \int_{0 \leq t_1 \leq \dots \leq t_k \leq 1} f^{(k)}(0) \,dt_1 \dots dt_k + \int_{0 \leq t_1 \leq \dots \leq t_{m+1} \leq 1} f^{(m+1)}(t_1)\,dt_1 \dots dt_{m+1}.
	\]
	Thus, the $1/k!$ represents the volume of the $k$-simplex $1 \leq t_1 \leq t_2 \leq \dots \leq t_k \leq 1$.  We can then parameterize the simplex by the map
	\[
	F: [0,1]^k \to \{0 \leq t_1 \leq \dots \leq t_k \leq 1\}, \qquad
	(s_1,\dots,s_k) \mapsto (s_1 \dots s_k, \, s_2 \dots s_k, \, \dots, \, s_k),
	\]
	and this change of variables leads to
	\begin{multline*}
		f(1) = \sum_{k=0}^m \int_{[0,1]^k} f^{(k)}(0) s_1^{k-1} s_2^{k-2} \dots s_{k-1}^1 \,ds_1 \dots ds_k \\ + \int_{[0,1]^{m+1}} f^{(m+1)}(s_1 \dots s_{m+1})\,s_1^m s_2^{m-1} \dots s_m^1 \,ds_1 \dots ds_{m+1}.
	\end{multline*}
	Corollary \ref{cor: iterated Parraud} has a similar form to this version of Taylor's theorem, but the integration is over $[0,1]^{2k}$ rather than $[0,1]^k$, and the integrand depends in a more complicated way on the parameters $t_1$, \dots, $t_k$ and $s_1$, \dots, $s_k$, where $(s_j,t_j)$ are the parameters introduced in the $k+1-j$th iteration.   Meanwhile, in Parraud's original formula, the parameters $s_j$ and $t_j$ are integrated over a region defined by a family of inequalities, namely the set $A_i$ in \cite[Proposition 3.6]{Parraud2023}, which corresponds in our analogy to the simplex in our statement of Taylor's formula.
	
	Finally we remark that the appearance of a fourth-order differential operator in Parraud's formula, which may seem surprising at first, is actually natural from the viewpoint of the genus expansion for GUE.  Indeed, for the trace of a non-commutative monomial of degree $2m$, the highest genus that can arise in this formula is $\lfloor m/2 \rfloor$ (for proof, see Remark \ref{rem: max genus}).  Hence, the genus expansion terminates after $\lfloor m/2 \rfloor$ terms.  Meanwhile, since $T^{\crss}$ reduces the degree of any monomial by $4$, we see that $\lfloor m/2 \rfloor + 1$ applications of $T^{\crss}$ will annihilate any polynomial of degree $2m$, and hence Parraud's formula also terminates after $\lfloor m/2 \rfloor$ terms.

	\subsection*{Acknowledgements}
	
	This work was partly supported by the National Science Foundation (US), grant DMS-2002826 in 2020-2023; the National Sciences and Engineering Research Council (Canada), grant RGPIN-2017-05650 in 2023-2024; and Denmark's Independent Research Fund in 2024.
	
	The Fields Institute, where part of this work was done, is on the traditional land of the Huron-Wendat, the Seneca, and the Mississaugas of the Credit.
	
	I thank Ian Charlesworth, Ben Hayes, Jamie Mingo, Akihiro Miyagawa, F{\'e}lix Parraud, Dimitri Shlyakhtenko, Roland Speicher for discussions about this work.  I thank the referee for feedback that improved the exposition of the paper.
	
	\section{Preliminaries} \label{sec: preliminaries}
	
	\subsection{GUE random matrices}
	
	$\M_N = M_N(\C)$ denotes the $N \times N$ complex matrices.  We write $\tr_N = (1/N) \Tr_N$ for the normalized trace.  We view $\M_N$ as an inner product space with the inner product given by $\ip{X,Y} = \tr_N(X^*Y)$.  Let $(\M_N)_{\sa}$ be the real subspace of self-adjoint matrices.  The inner product is real on $(\M_N)_{\sa}$, and $\M_N$ can be canonically identified with the complexification $\C \otimes_{\R} (\M_N)_{\sa}$.
	
	Since $(\M_N)_{\sa}$ is a real inner-product space of dimension $N^2$ and hence is isomorphic as a real inner-product space to $\R^{N^2}$.  Thus, there is a canonical choice of Lebesgue measure on $(\M_N)_{\sa}$ given by transporting via an isometry.  Moreover, the gradient and divergence operations for functions on $(\M_N)_{\sa}$ are defined with respect to the inner-product structure, or equivalently by transporting the gradient and divergence operations from $\R^{N^2}$.  It is elementary to check that these definitions are independent of the choice of linear isometry $(\M_N)_{\sa} \to \R^{N^2}$.  However, it is sometimes convenient to use the basis
	\[
	\mathcal{B} = \{ \sqrt{N} E_{j,j}\}_{1 \leq j \leq N} \cup \{\sqrt{N/2}(E_{j,k} + E_{k,j})\}_{1 \leq j < k \leq N} \cup \{\sqrt{N/2}(iE_{j,k} - iE_{k,j})\}_{1 \leq j < k \leq N}.
	\]
	
	\begin{definition}
		A \emph{standard $N \times N$ GUE random matrix} is a random variable $X$ with values in $(\M_N)_{\sa}$ which has probability density on $(\M_N)_{\sa}$ proportional to $e^{N^2 \tr_N(X^2)}$.  Equivalently, $X$ is self-adjoint and
		\[
		(\sqrt{N} X_{j,j})_{j=1}^N \cup (\sqrt{2N} \re(X_{i,j}))_{i<j} \cup (\sqrt{2N} \im(X_{i,j}))_{i<j}
		\]
		are independent standard normal random variables.
	\end{definition}
	
	Next, we recall the following results on integration by parts for GUE matrices, which we will use in the proof of the genus expansion.  These facts are well-known in random matrix theory and the arguments are similar to \cite{Rains1997}, \cite{Cebron2013}, \cite[\S 3]{DHK2013}, \cite[\S 3.20]{GuionnetParkCity}, \cite[\S 14.1]{JekelThesis}.  They are also closely related to the Wick formula; see e.g.\ \cite[\S 1.5]{MingoSpeicher}.
	
	\begin{fact} \label{fact: IBP}
		Let $f: (\M_N)_{\sa} \to \M_N$ be a polynomial (or more generally a smooth function such that $f$ and its first derivatives grow polynomially at $\infty$), and let $X^{(N)}$ be an $N \times N$ standard GUE matrix.  Then
		\[
		\E \ip{X^{(N)},f(X^{(N)})} = \frac{1}{N^2} \E \Div[f(X^{(N)})].
		\]
	\end{fact}
	
	\begin{fact} \label{fact: polynomial divergence}
		Let $f: (\M_N)_{\sa} \to \M_N$ be given by
		\[
		F(X) = A_0XA_1 \dots X A_k,
		\]
		where $A_0$, \dots, $A_k$ are deterministic matrices.  Then
		\[
		\frac{1}{N^2} \Div[F(X)] = \sum_{j=1}^k \tr_N(A_0 X A_1 \dots XA_{j-1}) \tr_N(A_jXA_{j-1} \dots XA_k),
		\]
		Moreover, let $\phi(X) = \tr_N(F(X))$.  Then
		\[
		\nabla \phi(X) = \sum_{j=1}^k A_jXA_{j-1} \dots XA_kA_0 X A_1 \dots XA_{j-1}.
		\]
	\end{fact}
	
	Here $\Div(F(X))$ corresponds to $\tr_N \otimes \tr_N[\partial_X F]$ where $\partial_X$ is Voiculescu's free difference quotient.  Moreover, $\nabla \phi(X) = D_X^\circ F(X)$, where $D_X^\circ$ is Voiculescu's cyclic derivative as in Notation \ref{not: cyclic derivative}.  (However, here we only have one variable rather than variables indexed by a vector space.)
	
	\subsection{Free semicircular families}
	
	\begin{definition}
		A tracial von Neumann algebra is a pair $(\cM,\tau)$ where $\cM \subseteq B(H)$ that contains $1$ and is closed under addition, multiplication, adjoints, and limits in the weak operator topology; and $\tau: \cM \to \C$ is a linear functional satisfying for $x, y \in \cM$,
		\begin{enumerate}[(1)]
			\item $\tau(1) = 1$,
			\item $\tau(x^*x) \geq 0$,
			\item $\tau(x^*x) = 0$ if and only if $x = 0$,
			\item $\tau(xy) = \tau(yx)$,
			\item $\tau$ is continuous on the unit ball with respect to the weak operator topology.
		\end{enumerate}
	\end{definition}
	
	Although von Neumann algebras are the natural non-commutative analog of measure spaces, and hence a natural setting for free probability, the deeper theory of von Neumann algebras will play little role in this paper.  The main properties we will use here are that $\cM$ is a $*$-algebra and $\tau: \cM \to \C$ is a linear functional that $\tau(1) = 1$ and $\tau(xy) = \tau(yx)$ (which leads to cyclic symmetry of the traces of monomials).
	
	We next summarize some basic facts about free products.  For background, see e.g.\ \cite{voiculescu1992free-random,AGZ2009,MingoSpeicher}.
	
	\begin{definition}
		Let $(\cM,\tau)$ be a tracial von Neumann algebra.  Let $\cA_1$, \dots, $\cA_m$ be $*$-subalgebras.  We say that $\cA_1$, \dots, $\cA_m$ are \emph{freely independent} if whenever $k \geq 1$ and $i_1$, \dots, $i_k$ are indices in $[m]$ with $i_1 \neq i_2 \neq \dots \neq i_k$, whenever $a_j \in \cA_{i_j}$ for $j = 1$, \dots, $k$ with $\tau(a_j) = 0$, then $\tau(a_1 \dots a_k) = 0$.
		
		Moreover, we say that families $(x_{1,1},\dots,x_{1,n_1})$, \dots, $(x_{m,1},\dots,x_{m,n_m})$ in $M$ are freely independent if the respective $*$-algebras they generate are freely independent.
	\end{definition}
	
	\begin{fact}
		Given tracial von Neumann algebras $(\cM_1,\tau_1)$, \dots, $(\cM_m,\tau_m)$, there exists a tracial von Neumann algebra $(\cM,\tau)$, called their \emph{free product}, and trace-preserving inclusions $\iota_j: \cM_j \to \cM$ such that $\iota_1(\cM_1)$, \dots, $\iota_m(\cM_m)$ are freely independent and generate $\cM$.
	\end{fact}
	
	Another fact about free independence is the ``associative'' property.
	
	\begin{fact}
		If $\cA_1$, \dots, $\cA_m$ are freely independent in $(M,\tau)$, then $\cA_m$ is freely independent of the algebra generated by $\cA_1$, \dots, $\cA_{m-1}$.
	\end{fact}
	
	\begin{definition}
		If $(\cM,\tau)$ is a tracial von Neumann algebra and we say that $X \in M$ is \emph{standard semicircular variable} if $X$ is self-adjoint and
		\[
		\tau(p(X)) = \frac{1}{2 \pi} \int_{-2}^2 p(x) \sqrt{4 - x^2}\,dx
		\]
		for every polynomial $p$.
	\end{definition}
	
	The semicircular random variable serves as an analog in free probability theory of the standard normal random variable, and as a model for the large-$N$ behavior of the $N \times N$ GUE matrix $X^{(N)}$ via Wigner's semicircle law.  We will use a free analog of the integration by parts formula that holds for Gaussian random variables  The next fact follows from \cite[Proposition 3.6 and Proposition 3.8]{VoiculescuFE5}, and a more self-contained proof is found in \cite[\S 3.5]{ShlyakhtenkoParkCity}.
	
	\begin{fact} \label{fact: free IBP}
		Let $(\cM,\tau)$ be a tracial von Neumann algebra.  Let $X$ be a standard semicircular random variable freely independent from the $*$-subalgebra $\cA$.  Then for $A_0, \dots, A_k \in \cA$, we have
		\[
		\tau(X \, A_0 X A_1 \dots X A_k) = \sum_{j=1}^k \tau(A_0 X A_1 \dots XA_{j-1}) \tau(A_jXA_{j-1} \dots XA_k).
		\]
	\end{fact}
	
	\section{Genus expansion for mixtures of GUE and semicirculars} \label{sec: genus expansion}
	
	\subsection{Permutations}
	
	Although the genus expansion is often stated in terms of maps or partitions, we find it technically convenient to use permutations; the relationship between partitions and permutations is explained for instance in \cite[\S 1.8]{MingoSpeicher}.
	
	\begin{notation}[Permutations, cycles, support]
		For a finite set $S$, $\Perm(S)$ will denote the finite permutations of the set $S$.  For $\sigma \in \Perm(S)$, we write
		\[
		\supp(\sigma) := \{k \in S: \sigma(k) \neq k\}
		\]
		Every permutation has a unique decomposition into a product of disjoint cycles.  A \emph{nontrivial cycle} is any cycle of length at least two (which can be represented as a permutation of $S$), and a \emph{trivial cycle} is a fixed point, which we represent simply as a point of $S$.  Let
		\[
		\Cyc(\sigma) := \{ \gamma \in \Perm(S): \gamma \text{ is a nontrivial cycle of } \sigma \} \sqcup \{k \in S: \sigma(k) = k\} \subseteq \Perm(S) \sqcup S.
		\]
		Thus, $\# \Cyc(\sigma)$ will denote the number of cycles of $\sigma$.  If $S' \subseteq S$ is a union of some cycles of $\sigma$, then we denote by $\sigma|_{S'}$ the permutation of $S'$ obtained by restricting $\sigma$.
	\end{notation}

	\begin{definition}[Labelings and compatibility] \label{def:labeling}
		Given a set $S$ and another set $I$, an \emph{$I$-labeling of $S$} is a function $\ell: S \to I$.   For an $I$-labeling $\ell$ of $S$ and $\sigma \in \Perm(S)$, we say that  \emph{$\sigma$ and $\ell$ are compatible} if $\ell \circ \sigma = \ell$.
	\end{definition}
	
	\begin{notation}[Pair permutations]
		For $S \subseteq \N$, let
		\[
		\Perm_2(S) := \{\sigma \in \Perm(S): \sigma^2 = \id, \supp(\sigma) = S  \}.
		\]
		Moreover, we denote by $\Trans(S)$ the set of transpositions in $\Perm(S)$.
	\end{notation}
	
	\begin{definition}[Bridges and crossings]
		Let $\sigma \in \Perm(S)$.
		\begin{itemize}
			\item A \emph{$\sigma$-bridge} is a transposition $\tau \in \Perm(S)$ which swaps two indices $k$ and $k'$ which are in different cycles (or fixed points) of $\sigma$.
			\item A \emph{$\sigma$-cross} is a pair of disjoint transpositions $\tau$, $\tau' \in \Perm(S)$ such that the following conditions hold: There is some $\gamma \in \Cyc(\sigma)$ such that $\supp(\tau)$ and $\supp(\tau') \subseteq \supp(\gamma)$.  Moreover, writing $\gamma$ in cycle notation as $(k_1 \dots k_m)$, there are indices $1 \leq a < b < c < d \leq m$ such that $\tau \tau'$ swaps $a$ with $c$ and $b$ with $d$.
		\end{itemize}
	\end{definition}
	
	\begin{definition}[$\sigma$-noncrossing partitions] \label{def: noncrossing}
		Let $S' \subseteq S$ be finite subsets of $\N$.  Let $\sigma \in \Perm(S)$.  We say that $\pi \in \Perm_2(S')$ is \emph{$\sigma$-noncrossing} if $\Cyc(\pi)$ does not contain any $\sigma$-bridges or $\sigma$-crosses.
	\end{definition}
	
	\begin{figure}
		
		\begin{center}
			
			\begin{tikzpicture}
				
				\draw (-3,0) circle (1.5);
				\node at (-1.8,0) {$1$};
				\node at (-3,1.2) {$2$};
				\node at (-4.0,0.7) {$3$};
				\node at (-4.0,-0.7) {$4$};
				\node at (-3,-1.2) {$5$};
				\draw (3,0) circle (1);
				\node at (3.2,0.7) {$6$};
				\node at (2.3,0) {$7$};
				\node at (3.2,-0.7) {$8$};
				
				\draw[dotted] (-1.5,0) -- (2,0);
				\draw[dotted] (3.3,0.95) arc (135:-135:1.35);
				\draw[dotted] (-3,1.5) arc (30:275:1.6);
				\draw[dotted] (-3,-1.5) arc (-30:-275:1.6);
				
				
			\end{tikzpicture}
			
		\end{center}
		
		\caption{Circles (solid) constructed from a permutation $\sigma = (12345)(678)$ and connecting curves (dotted) constructed from a permutation $\pi = (17)(24)(35)(68)$.  Here $(24)$ and $(35)$ are a $\sigma$-cross and $(17)$ is a $\sigma$-bridge.} \label{fig: bridge and cross}
		
	\end{figure}
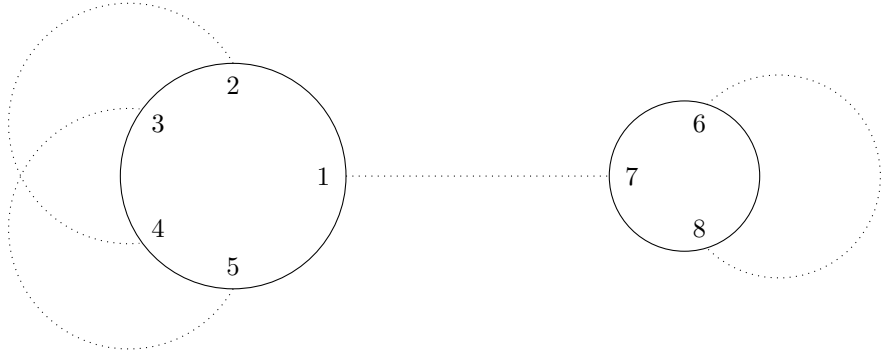
	
	\begin{remark}
		The diagrammatic intuition of Definition \ref{def: noncrossing} is as follows.  For each cycle of $\sigma$, imagine a circle with the indices of the cycle written counterclockwise around the circumference.  The cycles of $\pi$ are pairs or singletons.  For each pair in $\pi$, connect the corresponding indices on the circumferences of the circles by a curve that remains on the exterior of the circles.  A $\sigma$-bridge is represented by a curve that connects two different circles.  A $\sigma$-cross consists of two curves that connect indices on the same circle, and which cannot be drawn in the exterior of the circle without crossing each other.  See Figure \ref{fig: bridge and cross}.  Compare also \cite[\S 2.5]{GMS2006}.
	\end{remark}
	
	\begin{definition}
		Given permutations $\pi, \sigma \in \Perm(S)$, we define the permutation $\pi \bowtie \sigma \in \Perm(\{0,1\} \times S)$ by
		\begin{align*}
			(\pi \bowtie \sigma)(0,k) &= (1,\sigma(k)), \\
			(\pi \bowtie \sigma)(1,k) &= (0,\pi(k)).
		\end{align*}
	\end{definition}
	
	\begin{definition}[Combinatorial genus] \label{def: genus}
		For two permutations $\pi$ and $\sigma \in \Perm(S)$, we define the combinatorial genus $g(\sigma,\pi)$ by
		\[
		2g(\sigma,\pi) = \# S + \# \Cyc(\sigma) - \# \Cyc(\pi) - \# \Cyc(\pi \bowtie \sigma).
		\]
	\end{definition}
	
	\begin{observation}
		For any $\sigma \in \Perm(S)$, we have $g(\sigma,\id) = 0$.
	\end{observation}
	
	\begin{proof}
		Note $\# \Cyc(\id) = \# S$.  Moreover, if $j_1, \dots, j_m$ is a cycle of $\sigma$, then $(0,j_1)$, $(1,j_1)$, \dots, $(0,j_m)$, $(1,j_m)$ is a cycle of $\id \bowtie \sigma$.  Thus, $\# \Cyc(\sigma \bowtie \id) = \# \Cyc(\sigma)$.
	\end{proof}
	
	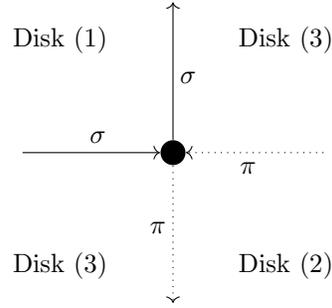
\begin{figure}
		
		\begin{center}
			
			\begin{tikzpicture}
				
				\node[circle,fill] (V) at (0,0) {};
				\draw[->] (-2,0) to (V);
				\draw[->] (V) to (0,2);
				\draw[->,dotted] (2,0) to (V);
				\draw[->,dotted] (V) to  (0,-2);
				\node at (-1,0.2) {$\sigma$};
				\node at (0.2,1) {$\sigma$};
				\node at (1,-0.2) {$\pi$};
				\node at (-0.2,-1) {$\pi$};
				\node at (-1.5,1.5) {Disk (1)};
				\node at (1.5,1.5) {Disk (3)};
				\node at (1.5,-1.5) {Disk (2)};
				\node at (-1.5,-1.5) {Disk (3)};
				
			\end{tikzpicture}
			
		\end{center}
		
		\caption{Local picture of the construction of a surface as in Remark \ref{rem: surface genus}.  The solid edges come from the permutation $\sigma$ while the dotted edges come from the permutation $\pi$.  On each corner, we glue disks of types (1), (2), (3) as described in the remark.} \label{fig: surface gluing} 
		
	\end{figure}
	
	\begin{remark} \label{rem: surface genus}
		This is the genus of a CW complex constructed as follows.  Let the set of vertices be $S$.  For each $j$, add an edge from $j$ to $\sigma(j)$ and an edge from $j$ to $\pi(j)$.  Here note a transposition produces two edges (for instance, in Figure \ref{fig: bridge and cross}, each dotted line would be doubled to form two edges).  Next, glue on disks of four types:
		\begin{enumerate}[(1)]
			\item For each cycle of $\sigma$, the corresponding edges form a circle to which one glues the circumference of a disk.  The number of such disks is $\# \Cyc(\sigma)$.
			\item For each cycle of $\pi$, the corresponding edges form a circle to which one glues the circumference of a disk.  The number of such disks is $\# \Cyc(\pi)$.
			\item Next, form circles from alternating edges: first an edge from $\sigma$, then an edge from $\pi$, then an edge from $\sigma$, then an edge from $\pi$, ending when we reach the same point via an edge from $\pi$.  These circles correspond to the cycles of $\pi \bowtie \sigma$.  Here we glue $\# \Cyc(\pi \bowtie \sigma)$ many disks.
		\end{enumerate}
		Because each edge has two circles glued to it and each vertex has four circles glued to it, we see that this forms a surface $M$; see Figure \ref{fig: surface gluing}.  The Euler characteristic can be computed as the number of vertices $\# S$ minus the number of edges $2 \# S$ plus the number of faces $\# \Cyc(\sigma) + \# \Cyc(\pi) + \# \Cyc(\pi \bowtie \sigma)$.  Recall that for a connected surface, the Euler characteristic is $2(1 - g)$ where $g$ is the genus.  In our situation, the surface may have multiple components; in fact, if $\pi$ is $\sigma$-noncrossing, then the number of components will be $\# \Cyc(\sigma)$; thus, we want to work with $\# \Cyc(\sigma) - g$ instead of $1 - g$.  The genus $g(\sigma,\pi)$ in Definition \ref{def: genus} satisfies
		\[
		2(\# \Cyc(\sigma) - g(\sigma,\pi)) = -\# S + \# \Cyc(\sigma) + \# \Cyc(\pi) + \# \Cyc(\pi \bowtie \sigma) = \chi(M).
		\]
		An alternative formulation is as follows:  Let $M_1$, \dots, $M_k$ be the connected components of $M$.  Note $\chi(M) = \sum_{j=1}^k \chi(M_j)$.  Thus,
		\[
		g(\sigma,\pi) = \# \Cyc(\sigma) - \frac{1}{2} \sum_{j=1}^k \chi(M_j) = \sum_{j=1}^k g(M_j) + (\# \Cyc(\sigma) - k),
		\]
		i.e.\ we add the genera of the components, then add the difference between $\# \Cyc(\sigma)$ and the total number of components.
	\end{remark}
	
	\begin{remark}[Maximum possible genus] \label{rem: max genus}
	Suppose that $\sigma$ is the $2m$-cycle $(12\dots (2m))$ and that $\pi \in \Perm_2([2m])$.  We claim that the maximum possible value of the genus $g(\sigma,\pi)$ is $\lfloor m/2 \rfloor$.  To see this, first note that since $\#S = 2m$ and $\# \Cyc(\sigma) = 1$ and $\# \Cyc(\pi) = m$, Definition \ref{def: genus} yields
	\[
	2g(\sigma,\pi) = 2m + 1 - m - \# \Cyc(\pi \bowtie \sigma) = m + 1 - \# \Cyc(\pi \bowtie \sigma).
	\]
	For the case where $m$ is even, it suffices to note that $\# \Cyc(\pi \bowtie \sigma) \geq 1$, so that $2g(\sigma,\pi) \leq m$.  Now suppose $m$ is odd.  The permutation $\pi \bowtie \sigma$ can be expressed as $(\pi \sqcup \sigma) \rho$ where $\pi \sqcup \sigma$ is the permutation that applies $\pi$ on $\{0\} \times [2m]$ and applies $\sigma$ on $\{1\} \times [2m]$, and $\rho$ is the permutation swapping $(0,j)$ and $(1,j)$ for $j = 1, \dots, 2m$.  Thus,
	\[
	\sgn(\pi \bowtie \sigma) = \sgn(\pi) \sgn(\sigma) \sgn(\rho) = (-1)^m (-1) 1 = (1)^{m+1}
	\]
	since $\rho$ has $2m$ transpositions, $\pi$ has $m$ transpositions, and $\sigma$ is a $2m$-cycle.  Therefore, if $m$ is odd, then $\pi \bowtie \sigma$ is an even permutation, and hence $\pi \bowtie \sigma$ cannot have only one cycle since a $4m$-cycle would be odd.  Hence, $\# \Cyc(\pi \bowtie \sigma) \geq 2$ which results in $2g(\pi,\sigma) \leq m - 1$.  So in either case $g(\pi,\sigma) \leq \lfloor m/2 \rfloor$.
	
	The value $g(\sigma,\pi) = \lfloor m/2 \rfloor$ is achieved by the permutation $\pi$ that swaps $j$ and $m + j$ for $j = 1, \dots, m$.  Indeed, in the even case, $\pi \bowtie \sigma$ has a single cycle
	\begin{align*}
	&(0,1) \mapsto (1,2) \mapsto (0,m+2) \mapsto (1,m+3) \mapsto  \dots \mapsto (0,m-1) \mapsto (1,m) \mapsto (0,2m) \mapsto (1,1) \mapsto \\
	 &(0,m+1) \mapsto (1,m+2) \mapsto (0,2) \mapsto (1,3) \mapsto \dots \mapsto (0,2m-1) \mapsto (1,2m) \mapsto (0,m) \mapsto (1,m+1) \mapsto \\
	 &(0,1),
	\end{align*}
	resulting in $2g(\sigma,\pi) = m$.  In the odd case, $\pi \bowtie \sigma$ has two cycles, namely
	\begin{align*}
		&(0,1) \mapsto (1,2) \mapsto (0,m+2) \mapsto (1,m+3) \mapsto  \dots \mapsto (0,m) \mapsto (1,m+1) \mapsto (0,1), \\
		&(0,m+1) \mapsto (1,m+2) \mapsto (0,2) \mapsto (1,3) \mapsto \dots \mapsto (0,2m) \mapsto (1,1) \mapsto (0,m+1),
	\end{align*}
	resulting in $2g(\sigma,\pi) = m - 1$.
	\end{remark}
	
	\subsection{Statement of the genus expansion}
	
	\begin{notation}
		Let $A$ be an algebra, let $S$ be a finite set, and let $\gamma \in \Perm(S)$ be a cycle.  Write $\gamma = (k_1 \dots k_m)$ in cycle notation with $k_1 = \min \supp(\gamma)$.  Given elements $(x_k)_{k \in S}$ of the algebra $A$, write
		\[
		\prod_{k \in \gamma} x_k := x_{k_1} \dots x_{k_m}.
		\]
		Moreover, if $\gamma$ is a single point $j \in S$ (this case will arise because the cycle decomposition of a permutation may include singletons), then $\prod_{k \in \gamma} x_k := x_{k_j}$.
	\end{notation}
	
	\begin{notation} \label{not:setup}
		Fix $N \in \N$.  Let $I$ and $J$ be index sets and consider the disjoint union $I \sqcup J$.  Let $(X_i^{(N)})_{i \in I}$ be a family of independent GUE matrices, which are random elements of $(\M_N)_{\sa}$.  Let $(X_j)_{j \in J}$ be a family of freely independent standard semicircular random variables and let $\cM$ be the tracial von Neumann algebra that they generate.  Let $\M_N * \cM$ be the free product.  We may view $X_i^{(N)}$ as a random element of $\M_N * \cM$, and similarly view $X_j$ as a deterministic element of $\M_N * \cM$.
	\end{notation}
	
	\begin{proposition} \label{prop:genusexpansion}
		With the setup of Notation \ref{not:setup}, let $S \subseteq \N$ be a finite set.  Let $\ell: S \to I \sqcup J$ be an $I \sqcup J$-labeling, and let $\sigma \in \Perm(S)$ Let
		\[
		Y_k = \begin{cases} X_{\ell(k)}^{(N)}, & \ell(k) \in I \\ X_{\ell(k)}, & \ell(k) \in J, \end{cases}
		\]
		and for $k \in S$, let $Z_k \in \M_N$ be a deterministic matrix.  Then
		\[
		\E \left[ \prod_{\gamma \in \Cyc(\sigma)} \tr_{\M_N * \cM} \left( \prod_{k\in \gamma} Y_kZ_k \right) \right] = \sum_{\substack{\pi \in \Perm_2(S) \\ \pi \text{\normalfont~and } \ell \text{\normalfont~compatible} \\ \pi|_{\ell^{-1}(J)} \sigma\text{\normalfont -noncrossing} }} \frac{1}{N^{2g(\sigma,\pi)}} \prod_{\gamma \in \Cyc(\pi \sigma)} \tr_N\left( \prod_{k \in \gamma} Z_k \right).
		\]
	\end{proposition}
	
	We will give a proof by induction on the number of Gaussians and semicirculars.  For the inductive step, we will want to remove one transposition at a time from the permutation $\pi$ through integration by parts (compare \cite[\S 2.6]{GMS2006}).  To this end, we will describe in the next subsection the effect of removing one transposition on all the objects in the genus expansion formula. The argument also motivates and aids our combinatorial proof of Parraud's formula in the next section.
	
	\subsection{Operations on permutations}
	
	\begin{notation}
		Let $\sigma, \tau \in \Perm(S)$.  Define $\sigma \take \tau \in \Perm(S \setminus \supp(\tau))$ as follows.  Let $r(k)$ be the smallest index $\geq 1$ such that $(\tau \sigma)^{r(k)} \not \in \supp(\tau)$.  Such an $r(k)$ must exist because $k \not \in \supp(\tau)$ and $\tau \sigma$ has finite order.  Then define $(\sigma \take \tau)(k) = (\tau \sigma)^{r(k)}(k)$.
	\end{notation}
	
	\begin{remark}
		Intuitively, $\sigma \take \tau$ represents the permutation obtained by collapsing the indices of $\tau$ in a certain fashion.  We will mainly be focused on the case where $\tau$ is a transposition.  Then if $\tau$ connects two points on the same ``circle'' from $\sigma$, $\sigma \take \tau$ will be obtained by cutting the circle at those two points and gluing the pieces into two circles (Figure \ref{fig: sigma take tau 1}).  But if $\tau$ connects two different circles, then we cut each of those circles at the corresponding point and patch them together to create $\sigma \take \tau$ (Figure \ref{fig: sigma take tau 2}).
	\end{remark}
	
	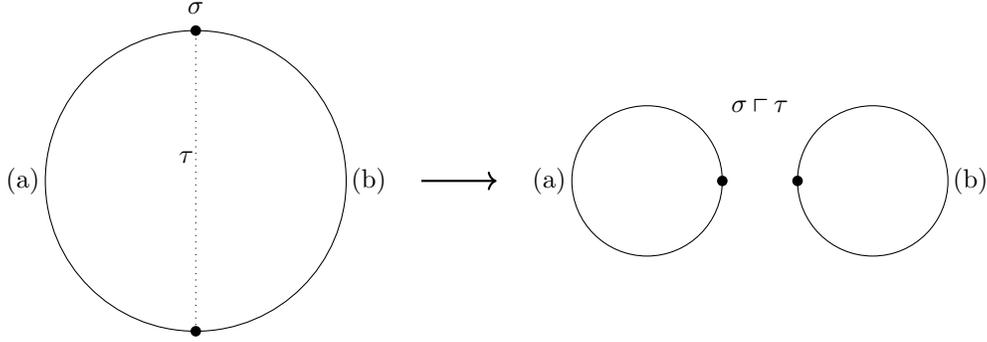
\begin{figure}
		
		\begin{center}
			
			\begin{tikzpicture}
				
				\draw (0,0) circle (2);
				\node at (0,2.3) {$\sigma$};
				\draw[dotted] (0,-2) to node[auto, label=above: $\tau$] {} (0,2);
				\node at (-2.3,0) {(a)};
				\node at (2.3,0) {(b)};
				\fill (0,2) circle (0.07);
				\fill (0,-2) circle (0.07);
				
				\draw[->,thick] (3,0) -- (4,0);
				
				\draw (6,0) circle (1);
				\node at (4.7,0) {(a)};
				\draw (9,0) circle (1);
				\node at (10.3,0) {(b)};
				\fill (7,0) circle (0.07);
				\fill (8,0) circle (0.07);
				\node at (7.5,1) {$\sigma \take \tau$};
				
			\end{tikzpicture}
			
		\end{center}
		
		\caption{A cycle $\sigma$ and a transposition $\tau$ on a set $S$ produce a permutation $\sigma \take \tau$ consisting of two cycles.  We label the two pieces (a) and (b).} \label{fig: sigma take tau 1}
		
	\end{figure}
	
	\begin{figure}
		
		\begin{center}
			
			\begin{tikzpicture}
				
				\draw (0,0) circle (2);
				\node at (0,2.3) {$\sigma \take \tau$};
				\node at (-2.3,0) {(a)};
				\node at (2.3,0) {(b)};
				\fill (0,2) circle (0.07);
				\fill (0,-2) circle (0.07);
				
				\draw[->,thick] (-4,0) -- (-3,0);
				
				\draw (-6,0) circle (1);
				\draw (-9,0) circle (1);
				\node at (-7.5,1) {$\sigma$};
				\draw[dotted] (-8,0) to node[auto,label=below: $\tau$] {} (-7,0);
				\node at (-4.7,0) {(b)};
				\node at (-10.3,0) {(a)};
				\fill (-7,0) circle (0.07);
				\fill (-8,0) circle (0.07);
				
			\end{tikzpicture}
			
		\end{center}
		
		\caption{A permutation $\sigma$ consisting of two cycles and a transposition $\tau$ connecting them produce a permutation $\sigma \take \tau$ consisting of one cycle.} \label{fig: sigma take tau 2}
		
	\end{figure}
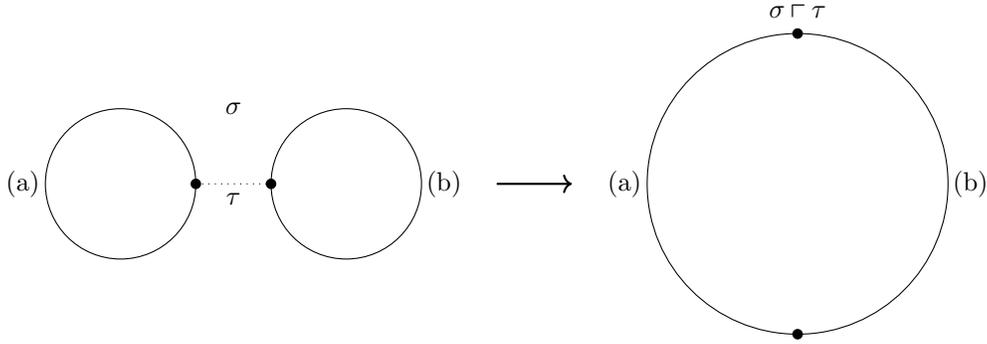
	
	\begin{lemma} \label{lem:takeaway}
		Let $S$ be a finite set, let $\sigma \in \Perm(S)$, and let $\tau$ and $\tau'$ be permutations of $S$ with disjoint supports.
		\begin{enumerate}[(1)]
			\item $(\sigma \take \tau)^{-1} = \sigma^{-1} \take \tau^{-1}$.
			\item $(\sigma \take \tau) \take \tau' = \sigma \take (\tau \tau') = (\sigma \take \tau') \take \tau$.
			\item Two indices $(i,k), (i',k') \in \{0,1\} \times (S \setminus \supp(\tau))$ are in the same cycle of $\tau' \tau \bowtie \sigma$ if and only if they are in the same cycle of $\tau' \bowtie (\sigma \take \tau)$.
		\end{enumerate}
	\end{lemma}
	
	\begin{proof}
		(1) Note that $\supp(\tau) = \supp(\tau^{-1})$ and $\supp(\sigma) = \supp(\sigma^{-1})$.  Let $k \in S \setminus \supp(\tau)$.  Then $(\sigma \take \tau)(k)$ is obtained by applying $\tau \sigma$ iteratively until we reach a point not in $\supp(\tau)$, and $r(k)$ is the number of iterations needed.  Since $k$ and $(\tau \sigma)^{r(k)} (k)$ are not in the support of $\tau$, we also have $(\sigma \tau)^{r(k)}(k) = (\tau \sigma)^{r(k)}(k)$.  Because a point $j$ is in $\supp(\tau)$ if and only if $\tau(j) \in \supp(\tau)$, we also have $(\sigma \tau)^{s}(k) \in \supp(\tau)$ for $0 < s < r(k)$.  Letting $k' = (\sigma \tau)^{r(k)}(k)$, we have $(\tau^{-1} \sigma^{-1})^s(k') = (\sigma \tau)^{r(k)-s}(k)$, which is not in $\supp(\tau)$ for any $0 < s < r(k)$.  Thus, $(\sigma^{-1} \take \tau^{-1})(k') = k$, and hence $(\sigma^{-1} \take \tau^{-1})^{-1}(k) = (\sigma \take \tau)(k)$ as desired.
		
		(2) By symmetry and the fact that $\tau \tau' = \tau' \tau$, it suffices to show that $(\sigma \take \tau) \take \tau' = \sigma \take (\tau \tau')$.  Let $k \in S \setminus \supp(\tau \tau')$.  Then $((\sigma \take \tau) \take \tau')(k)$ is obtained by repeatedly applying $\tau'(\sigma \take \tau)$ to $k$ until we find a point not in the $\supp(\tau')$.  Similarly, each application of $\sigma \take \tau$ within this procedure represents iterating $\tau \sigma$ until we arrive at a point not in $\supp(\tau)$.  Thus, the final point will neither be in $\supp(\tau')$ nor in $\supp(\tau)$.  But every intermediate point obtained in this procedure will either be an output of $\tau'(\sigma \take \tau)$ that lands in $\supp(\tau')$ or an output of $\tau \sigma$ which lands in $\supp(\tau)$.  Therefore, the point $((\sigma \take \tau) \take \tau')(k)$ reached at the end of the procedure will be the first point not in $\supp(\tau \tau')$ reached by application of $\tau \tau' \sigma$, which means it is equal to $(\sigma \take (\tau \tau'))(k)$.
		
		(3) This follows by similar reasoning as (2).  Fix $k \in S' := S \setminus \supp(\tau)$.  Then the orbit of $(0,k)$ in $\tau' \bowtie (\sigma \take \tau)$ is obtained by alternately applying the operations $(0,x) \mapsto (1,(\sigma \take \tau)(x))$ and $(1,x) \mapsto (0,\tau'(x))$.  Each application of $\sigma \take \tau$ occurs by repeated alternation between $\sigma$ and $\tau$ until we reach a point that is not in the support of $\tau$.  Since the resulting point is fixed by $\tau$, we can write $(\sigma \take \tau)(x) = \sigma (\tau \sigma)^{r(x)-1}(x)$, or in other words, it is obtained by alternation between $\sigma$ and $\tau$ where the first and last terms are applications of $\sigma$.  There is a corresponding sequence of points obtained by iterating $\tau'\tau \bowtie \sigma$, namely
		\[
		(0,x) \mapsto (1,\sigma(x)) \mapsto (0,\tau \tau' \sigma(x)) \mapsto \dots \mapsto (1,(\tau \tau' \sigma)^{r(x)-1}) \mapsto (0,(\sigma \take \tau)(x)).
		\]
		From this we can see the correspondence of the cycles in $\tau' \tau \bowtie \sigma$ and $\tau' \bowtie (\sigma \take \tau)$.
	\end{proof}
	
	\begin{lemma} \label{lem:genusupdate}
		Let $\sigma$ and $\pi \in \Perm(S)$.  Let $\tau \in \Cyc(\pi)$ be a transposition.  Let $S' = S \setminus \supp(\tau)$ and $\pi' = \pi|_{S'}$.  Then
		\begin{equation} \label{eq:genius}
			g(\sigma,\pi)  - g(\sigma \take \tau, \pi') = \begin{cases} 1, & \tau \text{ is a $\sigma$-bridge} \\
				0, & \text{otherwise.}
			\end{cases}
		\end{equation}
	\end{lemma}
	
	\begin{proof}
		Note that $\# S' = \# S - 2$ and $\# \Cyc(\pi') = \# \Cyc(\pi) - 1$.  Therefore,
		\begin{multline} \label{eq:genuschange}
			2g(\sigma,\pi) - 2g(\sigma \take \tau, \pi') \\
			= 2 - 1 + \# \Cyc(\sigma \take \tau) - \# \Cyc(\sigma) - \# \Cyc(\pi \bowtie \sigma) + \# \Cyc(\pi' \bowtie (\sigma \take \tau)).
		\end{multline}
		Below we describe how to compute the $\# \Cyc(\sigma \take \tau) - \# \Cyc(\sigma)$ and $\Cyc(\pi \bowtie \sigma) - \# \Cyc(\pi' \bowtie (\sigma \take \tau))$ in various cases.  By plugging these results into \eqref{eq:genuschange}, the reader may verify \eqref{eq:genius}.
		
		\textbf{Case 1:} Suppose $\tau$ is a $\sigma$-bridge connecting indices $k_1$ and $\ell_1$ in two distinct cycles $(k_1 \dots k_m)$ and $(\ell_1 \dots \ell_n)$ of $\sigma$.
		
		\textbf{Subcase (a):} Suppose that $m$ and $n$ are not both $1$.  Then $\sigma \take \tau$ has a cycle $(k_2 \dots k_m \ell_2 \dots \ell_n)$.  Thus, we have $\# \Cyc(\sigma \take \tau) = \# \Cyc(\sigma) - 1$.  Moreover, by the previous lemma every orbit of $\pi' \bowtie (\sigma \take \tau)$ is contained in a corresponding orbit of $\pi \bowtie \sigma$.  Since at least one index $a$ out of $k_1$ and $\ell_1$ is inside a nontrivial cycle of $\sigma$, the cycles of $\pi \bowtie \sigma$ containing $(0,a)$ and $(1,a)$ have a nontrivial intersection with $\{0,1\} \times S'$; if the other index $b$ out of $k_1$ or $\ell_1$ is a fixed point of $\sigma$, then $(j,b)$ will be in the same cycle of $\pi \bowtie \tau$ with $(1-j,a)$.  Hence, every cycle of $\pi \bowtie \sigma$ has some elements in $\{0,1\} \times S'$, so the number of cycles of $\pi \bowtie \sigma$ and $\pi' \bowtie (\sigma \take \tau)$ are the same. By \eqref{eq:genuschange}, we get $g(\sigma,\pi) - g(\sigma \take \tau,\pi') = 1$.
		
		\textbf{Subcase (b):} Suppose $m = n = 1$.  Then both the points $k_1$ and $\ell_1$ are in the support of $\tau$, and in this case $\sigma \take \tau = \sigma|_{S'}$ and $\# \Cyc(\sigma \take \tau) = \# \Cyc(\sigma) - 2$ because two fixed points were removed.  Moreover, in this case, $(0,k_1)$, $(1,k_1)$, $(0,\ell_1)$, and $(1,\ell_1)$ form an orbit of $\pi \bowtie \sigma$.  Thus, $\pi' \bowtie (\sigma \take \tau)$ has one less orbit than $\pi \bowtie \sigma$.  Still, by \eqref{eq:genuschange}, we get $g(\sigma,\pi) - g(\sigma \take \tau,\pi') = 1$.
		
		\textbf{Case 2:} Suppose that $\tau$ is not a $\sigma$-bridge.  Let $\gamma$ be the cycle of $\sigma$ containing $\tau$.  We can write in cycle notation
		\[
		\gamma = (a k_1 \dots k_m b \ell_1 \dots \ell_n).
		\]
		Thus, $(k_1 \dots k_m) (\ell_1 \dots \ell_n)$ will be part of the cycle decomposition of $\sigma$.
		
		\textbf{Subcase (a):} Suppose that $m, n > 0$.  Then $\sigma \take \tau$ has one more cycle than $\sigma$.  Moreover, similar to the argument of Case 1(a), every cycle of $\pi \bowtie \sigma$ intersects $\{0,1\} \times S'$, so $\pi' \bowtie (\sigma \take \tau)$ has the same number of cycles as $\pi \bowtie \sigma$.
		
		\textbf{Subcase (b):} Suppose that exactly one of $m$ or $n$ is zero (by symmetry we may assume that $m = 0$).  Then $\sigma \take \tau$ has the cycle $(\ell_1 \dots \ell_n)$ instead of $\gamma$, but overall has the same number of cycles.  Now $(a,0)$ and $(b,1)$ will form an orbit in $\pi \bowtie \sigma$, but the orbit of $(a,1)$ will have $(\ell_1,1)$ and hence $\pi' \bowtie (\sigma \take \tau)$ has overall one less cycle than $\pi \bowtie \sigma$.
		
		\textbf{Subcase (c):} Suppose that $m = n = 0$.  Then $\sigma \take \tau$ has one less cycle than $\sigma$.  Moreover,$\pi \bowtie \sigma$ has two transpositions that swap $(0,a)$ with $(1,b)$ and $(1,a)$ with $(0,b)$.  Thus, $\# \Cyc(\pi' \bowtie (\sigma \take \tau)) = \# \Cyc(\pi \bowtie \sigma) - 2$.
		
		In all of these cases, we obtain \eqref{eq:genius}.
	\end{proof}
	
	When rewriting the expectation of a product indexed by $\sigma$ into a product indexed by $\sigma \take \tau$, we must also rewrite the corresponding term $\prod_{\gamma \in \Cyc(\pi \sigma)} \tr_N(\prod_{k \in \gamma} Z_k)$ in terms of $\pi|_{S \setminus \supp(\tau)}$ and $\sigma \take \tau$.  Thus, we make the following definition.
	
	\begin{notation} \label{not:updateZ}
		Given a finite set $S \subseteq \N$, permutations $\sigma, \tau \in \Perm(S)$ with $\supp(\tau) \neq S$, and $\bZ \in \M_N^S$, we define the \emph{$(\sigma,\tau)$-reduction} $\bZ^{\sigma,\tau} \in \M_N^{S \setminus \supp(\tau)}$ as follows.
		
		For each $j \in S \setminus \supp(\tau)$, we let $r(j) \geq 1$ be the first index such that $(\tau \sigma)^r(j) \not \in \supp(\tau)$ and then set
		\[
		Z_j^{\sigma,\tau} = Z_j Z_{(\tau \sigma)(j)} \dots Z_{(\tau \sigma)^{r(j)-1}(j)}.
		\]
		Note that if $r(j) = 1$, then $Z_j^{\sigma,\tau}$ simply equals $Z_j$.   There may also be some $j \in \supp(\tau)$ such that $(\tau \sigma)^r(j) \in \supp(\tau)$ for all $r \geq 1$, or in other words, $j$ is part of a cycle of $\tau \sigma$ that is fully contained in $\supp(\tau)$.  We take the scalar
		\[
		\lambda_{\sigma,\tau} = \prod_{\substack{\gamma \in \Cyc(\tau \sigma) \\ \supp(\gamma) \subseteq \supp(\tau)}} \tr_N\left( \prod_{k \in \gamma} Z_k \right);
		\]
		note that this may contain fixed points of $\tau \sigma$.  We then modify the value of $Z_m$, where $m = \min(S \setminus \supp(\tau))$, by multiplying it by $\lambda_{\sigma,\tau}$.
		
		In the case $\supp(\tau) = S$, the $(\sigma,\tau)$ reduction of $\mathbf{Z}$ is defined as the scalar $\lambda_{\sigma,\tau}$. All the cycles of $\tau \sigma$ are in $\supp(\tau)$ and hence contribute to the definition of $\lambda_{\sigma,\tau}$.
	\end{notation}
	
	We have defined $\bZ^{\sigma,\tau}$ in such a way to make the following statement hold.
	
	\begin{lemma} \label{lem:updateZ}
		Let $\sigma, \tau \in \Perm(S)$.  Set $S' = S \setminus \supp(\tau)$.  Let $\pi' \in \Perm(S')$, and let $\pi$ be the permutation $\pi|_{\supp(\tau)} = \tau$ and $\pi|_{S'} = \pi'$.    Let $\bZ \in \M_N^S$.  If $S' = \varnothing$, then $\prod_{\gamma \in \Cyc(\pi \sigma)} \tr_N(\prod_{k \in \gamma} Z_k) = \prod_{\gamma \in \Cyc(\tau \sigma)} \tr_N(\prod_{k \in \gamma} Z_k) = \lambda_{\sigma,\tau}$.  Otherwise, we have
		\[
		\prod_{\gamma \in \Cyc(\pi \sigma)} \tr_N \left( \prod_{k \in \gamma} Z_k \right) =	
		\prod_{\gamma \in \Cyc(\pi'(\sigma \take \tau))} \tr_N \left( \prod_{k \in \gamma} Z_k^{\tau,\sigma} \right).
		\]
	\end{lemma}
	
	\begin{proof}
		The argument is similar to the proof of Lemma \ref{lem:takeaway}.  Each cycle of $\pi' (\sigma \take \tau)$ corresponds to a cycle of $\pi \sigma$, and we have define $Z_k^{\sigma,\tau}$ so that the $Z_k$'s in the support of $\tau$ are absorbed into the $Z_k$'s for $k \in S'$ at the appropriate place in the cycle.  Moreover, the cycles of $\pi \sigma$ that do not appear in $\pi' (\sigma \take \tau)$, or those which are contained in $\supp(\tau)$, are absorbed into the scalar $\lambda_{\sigma,\tau}$.
	\end{proof}
	
	\subsection{Inductive proof of the genus expansion}
	
	Now let us prove Proposition \ref{prop:genusexpansion}.  Although the case of all Gaussians or all semicirculars is well known, we give a complete proof using the tools of this section together with integration by parts.  One of the goals is to illustrate the connection between Gaussian integration by parts and the operations on permutations (as was done for planar maps in \cite[\S 2.6]{GMS2006}), since Parraud's formula was first proved using integration by parts.
	
	\begin{proof}[Proof of Proposition \ref{prop:genusexpansion}]
		Fix $I$, $J$, $\ell$, and $Z_k$ as in the proposition.  First, consider the case where all the matrices are from the Gaussian Unitary Ensemble.  Let $m = \min(S)$, and let $\gamma = (m m_1 \dots m_n)$ be the cycle containing $m$.  Let $X^{(N)} = (X_i^{(N)})_{i \in I}$, and let
		\[
		F(X^{(N)}) = Z_m X_{\ell(m_1)}^{(N)} Z_{m_1} \dots X_{\ell(m_n)}^{(N)} Z_{m_n}
		\]
		and for $\gamma' \in \Cyc(\sigma) \setminus \gamma$, let
		\[
		\phi_{\gamma'}(X^{(N)}) = \tr_N\left( \prod_{k \in \gamma'} X_{\ell(k)}^{(N)} Z_k \right)
		\]
		so that
		\[
		\E \tr_N \left[ \prod_{\gamma' \in \Cyc(\sigma)} \tr_N\left( \prod_{k \in \gamma'} X_{\ell(k)}^{(N)} Z_k \right) \right] = \E \ip*{X_{\ell(m)}^{(N)}, F(X^{(N)}) \prod_{\gamma' \in \Cyc(\sigma) \setminus \gamma} \phi_{\gamma'}(X^{(N)})}_{\tr_N},
		\]
		where $\ip{A,B}_{\tr_N} = \tr_N(A^* B)$ for $A, B \in \M_N$ (recall that this restricts to a real inner product on the real subspace of self-adjoint matrices).  Using Gaussian integration by parts (Fact \ref{fact: IBP}), this is equal to
		\begin{multline*}
			\frac{1}{N^2} \E[\Div_{X_{\ell(m)}}(F \cdot \prod_{\gamma' \in \Cyc(\sigma) \setminus \gamma} \phi_{\gamma'})((X_i^{(N)}))] = \frac{1}{N^2} \E\biggl[\Div_{X_{\ell(m)}}(F)(X^{(N)}) \prod_{\gamma' \in \Cyc(\sigma) \setminus \gamma} \phi_{\gamma'}(X^{(N)}) \\
			+ \frac{1}{N^2}\sum_{\gamma' \in \Cyc(\sigma) \setminus \gamma} \tr_N \left[ F(X^{(N)}), \nabla_{X_{\ell(m)}} \phi_{\gamma'}(X^{(N)}) \right] \prod_{\gamma'' \in \Cyc(\sigma) \setminus \{\gamma,\gamma'\}} \phi_{\gamma''}(X^{(N)}) \biggr],
		\end{multline*}
		where the second inequality follows from the product rule for the divergence of a scalar-valued function times a vector-valued function (where in this case the vector space is $\M_N$).  The divergence of $F$, as a function of $X_i^{(N)}$, can be evaluated by Fact \ref{fact: polynomial divergence}:
		\[
		\frac{1}{N^2} \Div_{X_{\ell(m)}}(F(X^{(N)})) = \sum_{\substack{j \in [n] \\ \ell(j) = \ell(m)}} \tr_N(Z_m X_{\ell(m_1)}^{(N)} Z_{m_1} \dots X_{\ell(m_{j-1})} Z_{j-1}) \tr_N(Z_j X_{\ell(m_{j+1})} Z_{\ell(m_j)} \dots X_{\ell(m_n)} Z_{m_n}).
		\]
		The order of terms in this expression follows the cycles of $\sigma \take (m m_j)$.  From this we can compute that
		\[
		\frac{1}{N^2} \Div_{X_{\ell(m)}}(F(X^{(N)})) \prod_{\gamma' \in \Cyc(\sigma) \setminus \{\gamma\}} \phi_{\gamma'}(X^{(N)})
		= \sum_{\substack{j \in [n] \\ \ell(j) = \ell(m)}} \prod_{\beta \in \Cyc(\sigma \take (m m_j))} \tr_N\left( \prod_{k \in \beta} X_{\ell(k)}^{(N)}Z_k \right).
		\]
		Next, we consider the $\nabla_{X_{\ell(m)}} \phi_{\gamma'}$ terms.  Let $\gamma' = (p_1 \dots p_r)$.  By the second part of Fact \ref{fact: polynomial divergence}, applied to $\phi_{\gamma'}$ as a function of $X_{\ell(m)}^{(N)}$, we get
		\[
		\nabla_{X_{\ell(m)}} \phi_{\gamma'}(X^{(N)}) = \sum_{\substack{j \in [r] \\ \ell(p_j) = \ell(m)}} Z_{p_{j+1}}X_{\ell(p_{j+1})}^{(N)} Z_{p_{j+2}} \dots X_{\ell(p_r)}^{(N)} Z_{p_r} X_{\ell(p_1)}^{(N)} Z_{p_1} \dots X_{p_{j-1}}^{(N)} Z_{p_j}
		\]
		Therefore, taking the inner product on $\M_N$ (or more accurately the complex bilinear extension of the real inner product on $(\M_N)_{\sa}$) and using cyclic symmetry of the trace, we get
		\begin{multline*}
			\frac{1}{N^2} \tr_N \left[ F(X^{(N)}) \nabla_{X_{\ell(m)}} \phi_{\gamma'}(X^{(N)}) \right] \\
			= \frac{1}{N^2} \sum_{\substack{j \in [r] \\ \ell(p_j) = \ell(m)}} \tr_N(X_{\ell(p_1)}^{(N)} Z_{p_1} \dots X_{\ell(p_{j-1})}^{(N)} Z_{p_{j-1}} Z_m X_{\ell(m_1)}^{(N)} Z_{m_1} \dots X_{\ell(m_n)}^{(N)} Z_{p_j} Z_{m_n} X_{\ell(p_{j+1})}^{(N)} Z_{p_{j+1}} \dots X_{\ell(p_r)}^{(N)} Z_{p_r}).
		\end{multline*}
		The order of multiplication here matches the order of the cycle of $\sigma \take (m p_j)$ obtained from the two cycles $\gamma$ and $\gamma'$ of $\sigma$ (in the case where $\gamma$ and $\gamma'$ were singletons the product above reduces to a deterministic matrix).  After multiplying this by the other terms $\phi_{\gamma''}$ for $\gamma'' \neq \gamma, \gamma'$, we obtain
		\[
		\frac{1}{N^2} \prod_{\beta \in \Cyc(\sigma \take (m p_j))} \tr_N\left( \prod_{k \in \beta} X_{\ell(k)}^{(N)}Z_k \right).
		\]
		Overall, we have one term for each $q$ such that $\ell(q) = \ell(m)$, as $q$ ranges over the $m_j$'s and $p_j$'s.  Moreover, in the case of the $m_j$'s, the transposition $(m q)$ is not a $\sigma$-bridge, but in the case of the $p_j$'s it is a $\sigma$-bridge.  Overall, we get
		\[
		\sum_{\substack{q \in S \\ \ell(q) = \ell(m)}} \frac{1}{(N^2)^{\mathbf{1}_{(mq) \text{ is a $\sigma$-bridge}}}} \E \left[ \prod_{\beta \in \sigma \take (mq)} \tr_N \left( \prod_{k \in \beta} Z_k^{\sigma,(mq)} \right) \right].
		\]
		By inductive hypothesis, we can evaluate this as
		\[
		\sum_{\substack{q \in S \\ \ell(q) = \ell(m)}} \sum_{\substack{\pi' \in \Perm_2(S \setminus \{m,q\}) \\ \pi' \text{ compatible with } \ell|_{S \setminus \{m,q\}}}} \frac{1}{N^{2\cdot \mathbf{1}_{(mq) \text{ is a $\sigma$-bridge}}}} \frac{1}{N^{2g(\sigma \take (mq), \pi')}}.
		\]
		Letting $\pi$ be the permutation of $S$ formed by $\pi'$ and $(mq)$, we get $g(\sigma,\pi) = g(\sigma \take (mq),\pi') + \mathbf{1}_{(mq)\text{ is a $\sigma$-bridge}}$.  When we sum over the different choices of $q$, then $\pi$ will range over all pair permutations of $S$ compatible with $\ell$, because $m$ must be paired with some $q$ with the same color.  This completes the inductive argument for the GUE case.
		
		Next, consider the case of a mixture of GUE and semicircular matrices.  Let us proceed by induction on the number of semicircular terms in our monomial.  Suppose that there is at least one semicircular at position $m$, and let $\gamma = (m m_1 \dots m_n)$ be the cycle containing $m$.  We apply free integration by parts (Fact \ref{fact: free IBP}) to the semicircular variable $X_{\ell(m)}$ that is freely independent from the $*$-algebra generated by $\M_N$ and $(X_j)_{j \in J \setminus \{\ell(m)\}}$ to obtain
		\begin{multline*}
			\tr_{\M_N * \cM}(X_{\ell(m)} Z_m Y_{m_1} \dots Y_{m_n} Z_{m_n}) \\
			= \sum_{\substack{j \in [n] \\ \ell(m_j) = \ell(m))}} \tr_{\M_N * \cM}(Z_m Y_{m_1} \dots Y_{m_{j-1}} Z_{m_{j-1}}) \tr_{\M_N * \cM_J}(Z_{m_j} Y_{m_{j+1}} Z_{m_{j+1}} \dots Y_{m_n} Z_{m_n}).
		\end{multline*}
		From this we find that
		\[
		\E \left[ \prod_{\beta \in \Cyc(\sigma)} \tr_{\M_N * \cM} \left( \prod_{k \in \beta} Y_k Z_k \right) \right]
		= \sum_{\substack{j \in [n], \\ \ell(j) = \ell(m)}} \E \left[ \prod_{\beta \in \Cyc(\sigma \take (mm_j))} \tr_{\M_N * \cM} \left( \prod_{k \in \beta} Y_k Z_k^{\sigma,(mm_j)} \right) \right].
		\]
		By induction hypothesis, the expression for each $j$ can be evaluated as a sum over $\pi' \in \Perm_2(S')$ where $S' = S \setminus (mm_j)$ that are compatible with $\ell$ and such that $\pi'|_{\ell^{-1}(J)}$ is $\sigma$-noncrossing.  Let $\pi \in \Perm_2(S)$ be the permutation formed from $\pi'$ and $(mm_j)$.  Because $(mm_j)$ is not a $\sigma$-bridge, we have $g(\sigma,\pi) = g(\sigma \take (mm_j), \pi')$.  Moreover, if $\pi'|_{\ell^{-1}(J)}$ is $(\sigma \take (mm_j))$-noncrossing, then $g(\sigma \take (mm_j),\pi'|_{\ell^{-1}(J)}) = 0$, which implies that $g(\sigma, \pi|_{\ell^{-1}(J)}) = 0$ and hence $\pi|_{\ell^{-1}(J)}$ is $\sigma$-noncrossing.  Similar reasoning shows the converse direction.  Therefore, arguing similarly to the GUE case, the sum over $(mm_j)$ and sum over $\Perm_2(S \setminus \{m,m_j\})$ that we get from the inductive hypothesis combine to produce the sum over $\pi \in \Perm_2(S)$ that we wanted, so the induction step is complete.
	\end{proof}
	
	\section{Combinatorial proof of Parraud's formula} \label{sec: Parraud formula proof}
	
	\subsection{Interpolation between GUE and semicircular}
	
	In this section, we use all the notation from \S \ref{subsec: statement}.  As in Parraud's work, we interpolate between $X_V^{(N)}$ and a free semicircular family $X_V$.  Here we view $X_V^{(N)}$, $X_V$, and $X_W$ as elements of $\M_N * \cM_{V \oplus W}$.  We then set
	\begin{equation} \label{eq:etadefinition}
		\eta_f(t) = \E \tr_{\M_N * \cM_{V\oplus W}}[f(t^{1/2}X_V^{(N)}+ (1-t)^{1/2} X_V,X_W)].
	\end{equation}
	Note that
	\begin{align*}
		\eta_f(0) &= \tr_{\cM_{V \oplus W}}[f(X_V,X_W)] \\
		\eta_f(1) &= \E \tr_{\M_N * \cM_W}[f(X_V^{(N)},X_W)]
	\end{align*}
	
	\begin{observation} \label{obs:derivativeintegral}
		In order to prove Theorem \ref{thm:Parraud1}, it suffices to show that
		\begin{equation} \label{eq:maingoal}
			\eta_f'(t) = \frac{1}{N^2} \int_0^1 \mathbb{E} \circ \tr_{\M_N * \cM_{(V \oplus W) \otimes \R^6}} \left[ (\alpha_{s,t})_*[(1-t)D_{V,V}^{\crss} + D_{V,W}^{\crss}](f)(X_V^{(N)},X_{(V\oplus W) \otimes \R^6}) \right] \,ds.
		\end{equation}
	\end{observation}
	
	\begin{proof}
		Integrating \eqref{eq:maingoal} over $t \in [0,1]$ will yield
		\begin{multline*}
			\eta_f(1) - \eta_f(0) \\
			= \frac{1}{N^2} \int_0^1 \int_0^1 \mathbb{E} \circ \tr_{\M_N * \cM_{(V \oplus W) \otimes \R^6}} \left[ (\alpha_{s,t})_*[(1-t)D_{V,V}^{\crss} + D_{V,W}^{\crss}](f)(X_V^{(N)},X_{(V\oplus W) \otimes \R^6}) \right] \,ds\,dt.
		\end{multline*}
		Since the integrand is a polynomial in Gaussian random variables, deterministic matrices, and free semicirculars, we can take the integral with respect to $s$ and $t$ inside the expectation and trace, so that the right-hand side reduces to $E \tr_{\M_N * \cM_{(V \oplus W) \otimes \R^6}}[T_{V,V}^{\crss}f(X_V^{(N)},X_{(V \oplus W) \otimes \R^6})]$.
	\end{proof}
	
	Then we want to compute $\eta_f'(t)$.  Here is where we follow a different approach than Parraud and give a combinatorial argument that illuminates the relationship of this formula with the genus expansion.  To make the computation tractable, we first restrict to the case of a monomial.
	
	\begin{observation} \label{obs:specialf}
		To prove Theorem \ref{thm:Parraud1}, it suffices to check  \eqref{eq:maingoal} when $f \in \M_N\ip{V \oplus W}$ has the form
		\[
		f(x_V,x_W) = x_{u(1)} Z_1 x_{u(2)} \dots x_{u(r)} Z_r,
		\]
		for some $u: [r] \to V \sqcup W$ and $Z_1$, \dots, $Z_r \in \M_N$.
	\end{observation}
	
	\begin{proof}
		Since the operator $T_{V,V}^{\crss}$ is linear, it suffices to check the equality when $f \in \M_N\ip{V \oplus W}$ has the form
		\[
		Z_0 x_{u(1)} Z_1 x_{u(2)} \dots x_{u(r)} Z_r,
		\]
		for $u(1)$, \dots, $u(k) \in V \oplus W$ and $Z_0$, \dots, $Z_r \in \M_N$.  Again, by linearity, we can further assume that $u(j) \in V \sqcup W$.  Moreover, by cyclic symmetry, the trace of $f(X_V^{(N)},X_W)$ will be the same if we move $Z_0$ to the right-hand side, and hence we may absorb $Z_0$ into $Z_r$.
	\end{proof}
	
	In order to prove \eqref{eq:maingoal}, we want to expand both sides as sums over $\Perm_2([r])$ (or equivalently over pair partitions).  The first step is to express $\E \tr_{\M_N * \cM_W}(f(X_V^{(N)},X_W))$ combinatorially using Proposition \ref{prop:genusexpansion}.
	
	\begin{notation}
		We say a permutation $\pi \in \Perm([r])$ is \emph{compatible} with a partition $(A_1,\dots,A_k)$ of $[r]$ if each cycle of $\pi$ is contained in one of the $A_j$'s.  Equivalently, this means that $\pi$ is compatible with the labeling $\ell: [r] \to [k]$ given by $\ell|_{A_j} = j$ in the sense of Definition \ref{def:labeling}.
	\end{notation}
	
	\begin{notation}
		Given $f$ as in Observation \ref{obs:specialf}, we denote by $S \subseteq [r]$ the set of indices $i$ such that $u(i) \in W$; this notation applies throughout the remainder of this section.
	\end{notation}
	
	\begin{lemma} \label{lem:etaformula}
		Let $f$ be as in Observation \ref{obs:specialf}. Let $\sigma \in \Perm([r])$ be the cycle $\sigma = (12\dots r)$.  Then
		\[
		\eta_f(t) = \sum_{\substack{
				A, B \subseteq [r] \\
				A \sqcup B \sqcup S = [r]}} \sum_{\substack{\pi \in \Perm_2([r]) \\
				\pi \text{\normalfont~and } (A,B,S) \text{\normalfont~compatible,} \\
				\pi|_{B \cup S} \text{\normalfont $\sigma$-noncrossing} }}
		\frac{1}{N^{2g(\sigma,\pi)}}
		a(\pi) b(\pi)
		t^{|A|/2} (1 - t)^{|B|/2}
		\]
		where
		\begin{align*}
			a(\pi) &:= \prod_{(ij) \in \Cyc(\pi)} \ip{u(i),u(j)}_{V \oplus W}, \\
			b(\pi) &:= \prod_{\gamma \in \Cyc(\pi \sigma)} \tr_N\left( \prod_{k \in \gamma} Z_k \right).
		\end{align*}
	\end{lemma}
	
	\begin{proof}
		We can assume without loss of generality that $V \oplus W$ is finite-dimensional by restricting to the span of $u(1)$, \dots, $u(r)$.  Then because both sides of the equation are real-multilinear in $u(1)$, \dots, $u(r)$, it suffices to consider the case where the $u(j)$'s are picked from a given orthonormal basis $\{e_1, \dots, e_m\}$ for $V \oplus W$, obtained as the union of an orthonormal basis for $V$ and one for $W$.  In other words, $u(j) = e_{\ell(j)}$ for each $j = 1, \dots, r$.  Thus, $\ip{u(i),u(j)}_{V \oplus W} = \delta_{\ell(j),\ell(j)}$ for $i, j \in [r]$.  Hence, $a(\pi)$ will equal $1$ if $\pi$ is compatible with the labeling $\ell$ and $0$ otherwise.
		
		Let $S$ be the set of indices $j$ for which $u(j) \in W$.  For each index $j \in [r] \setminus S$ with $u(j) \in V$, we are substituting $t^{1/2} X_{u(j)}^{(N)} + (1 - t)^{1/2} X_{u(j)}$.  We thus expand the expected trace of $f$ using multilinearity of $f$ in the variables $X_V$.  Each term is represented by the choice of partition of $[r] \setminus S$ into $A$ and $B$, where $A$ represents the indices where we chose $t^{1/2} X_{u(j)}^{(N)}$ and $B$ represents the set of indices where we chose $(1 - t)^{1/2} X_{u(j)}$.  Once $A$ and $B$ have been chosen, pull out the factors of $t^{|A|/2}$ and $(1 - t)^{|B|/2}$, and then evaluate the expectation of the monomial using Proposition \ref{prop:genusexpansion}.  This results in a sum over $\pi \in \Perm_2([r])$ that respect the partition $A$, $B$, $C$, are compatible with $\ell$, and are $\sigma$-noncrossing when restricted to $B \cup S$.  Since the term $a(\pi)$ is zero when $\pi$ is not compatible with $\ell$, we may remove the restriction of $\ell$-compatibility without changing the value, and hence the formula is proved in this special case.
	\end{proof}
	
	Next, we differentiate this formula for $\eta_f$.
	
	\begin{lemma} \label{lem:etaderivative}
		With the notation of Lemma \ref{lem:etaformula}, we have
		\[
		\eta_f'(t) = \sum_{\tau \in \Trans([r])}
		\sum_{\substack{
				A, B \subseteq [r] \\
				A \sqcup B \sqcup S = [r] \\
				\supp(\tau) \subseteq A}}
		\sum_{\substack{
				\pi \in \Perm_2([r]) \\
				\pi \text{\normalfont~and } (A,B,S) \text{\normalfont~compatible} \\
				\pi|_{B \cup S} \text{\normalfont ~$\sigma$-noncrossing} \\
				\pi|_{B \cup S \cup \supp(\tau)} \text{\normalfont ~$\sigma$-crossing}}}
		\frac{1}{N^{2g(\sigma,\pi)}}
		a(\pi) b(\pi)
		t^{|A|/2-1} (1-t)^{|B|/2}.
		\]
	\end{lemma}
	
	\begin{proof}
		We can rewrite the expression in Lemma \ref{lem:etaformula} as follows.  For each $A$, $B$, and $\pi$, we have the term
		\[
		\frac{1}{N^{2g(\sigma,\pi)}} a(\pi) b(\pi) t^{|A|/2} (1 - t)^{|B|/2}.
		\]
		We rewrite
		\[
		t^{|A|/2} (1 - t)^{|B|/2} = t^{\#\Cyc(\pi|_A)} (1 - t)^{\# \Cyc(\pi|_B)} = \left( \prod_{\tau \in \Cyc(\pi|_A)} t \right) \left( \prod_{\tau \in \Cyc(\pi|_B)} (1 - t) \right).
		\]
		Differentiating this with the product rule produces a sum of terms indexed by $\Cyc(\pi|_{A \cup B})$.  For each $\tau \in \Cyc(\pi|_A)$, the power of $t$ decreases by $1$, while for each $\tau \in \Cyc(\pi|_B)$, the power of $(1-t)$ decreases by $1$ and a negative sign appears.  Therefore, the derivative is
		\[
		\sum_{\tau \in \Cyc(\pi|_{A \cup B})} (\mathbf{1}_{\supp(\tau) \subseteq A} - \mathbf{1}_{\supp(\tau) \subseteq B}) t^{\# (\Cyc(\pi|_A) \setminus \{\tau\})} (1 - t)^{\# (\Cyc(\pi|_B) \setminus \{\tau\})}
		\]
		Summing up over $A$, $B$, $\pi$, and $\tau$ and exchanging the order of summation, we obtain
		\begin{multline*}
			\sum_{\tau \in \Trans([r])} \sum_{\substack{\pi \in \Perm_2([r]) \\ \tau \in \Cyc(\pi)}} \sum_{\substack{A, B \subseteq [r] \\ A \sqcup B \sqcup S = [r] \\ \pi \text{ compatible with } (A,B,S) \\ \pi|_{B \cup S} \text{ $\sigma$-noncrossing}}} \frac{1}{N^{2g(\sigma,\pi)}}
			a(\pi) b(\pi) \\
			(\mathbf{1}_{\supp(\tau) \subseteq A} - \mathbf{1}_{\supp(\tau) \subseteq B}) t^{\# (\Cyc(\pi|_A) \setminus \{\tau\})} (1 - t)^{\# (\Cyc(\pi|_B) \setminus \{\tau\})}.
		\end{multline*}
		If we fix $\tau$, $\pi$, and $A \setminus \supp(\tau)$ and $B \setminus \supp(\tau)$, then there are at most two choices of $(A,B)$ based on whether $\supp(\tau) \subseteq A$ or $\supp(\tau) \subseteq B$.  Note that $\pi|_{B \setminus \supp(\tau)}$ is required to be $\sigma$-noncrossing; but if $\pi|_{B \cup \supp(\tau)}$ is also $\sigma$-noncrossing, then there is one term with $\supp(\tau) \subseteq B$ and another term with $\supp(\tau) \subseteq A$, and these two terms have opposite signs, and hence they cancel.  The only remaining terms are those for which $\supp(\tau) \subseteq A$ and $\pi|_{B \cup \supp(\tau)}$ is $\sigma$-crossing.  Taking only those terms and then rearranging the order of summation yields the formula asserted by the lemma.
	\end{proof}

	\subsection{Combinatorial evaluation of the crossing derivative}
	
	Now we describe how to evaluate the right-hand side of \eqref{eq:maingoal}.  We apply each of the operations in Definition \ref{def: crossing derivative} in order, emphasizing how the operations translate into the combinatorics of permutations.  We can compute
	\begin{multline} \label{eq:Laplacianexpansion}
		\frac{1}{2} L_V f(x_V, x_W) =
		\frac{1}{2} \sum_{\substack{i_1 < i_2 \\ u(i_1), u(i_2) \in V}} \ip{u(i_1),u(i_2)} \\
		\biggl( Z_{j_1}x_{u(i_1+1)} Z_{i_1+1} \dots x_{u(i_2-1)} Z_{i_2-1} \otimes Z_{j_2} x_{u(i_2+1)} \dots x_{u(r)}Z_r x_{u(1)} Z_1 \dots x_{u(i_1-1)} Z_{i_1-1} \\
		+ Z_{i_2} x_{u(i_2+1)} \dots x_{u(r)}Z_r x_{u(1)} Z_1 \dots x_{u(i_1-1)} Z_{i_1-1} \otimes Z_{i_1}x_{u(i_1+1)} Z_{i_1+1} \dots x_{u(i_2-1)} Z_{i_2-1} \biggr)
	\end{multline}
	In other words, $L_V f(x_V,x_W)$ takes any pair of indices $i_1 < i_2$ and puts the letters in the word between $i_1$ and $i_2$ (from left to right cyclically) on the left side of the tensor sign, and the letters between $i_2$ and $i_1$ (from left to right cyclically) on the right side of the tensor sign; then it does the same for the pair $i_2 < i_1$.
	
	This can be described in terms of permutations as follows.  Fixing two indices $i_1$ and $i_2$, let $\tau$ be the transposition $(i_1 i_2)$.  Then write
	\[
	\sigma = (i_1 k_1 \dots k_m i_2 \ell_1 \dots \ell_n).
	\]
	Then the corresponding term in the sum is
	\begin{equation} \label{eq:derivativeterm}
		\ip{u(i_1),u(i_2)} Z_{i_1} x_{u(k_1)} Z_{k_1} \dots x_{u(k_m)} Z_{k_m} \otimes Z_{i_2} x_{u(\ell_1)} Z_{\ell_1} \dots x_{u(\ell_n} Z_{\ell_n}.
	\end{equation}
	Note that $\sigma \take \tau = (k_1 \dots k_m)(\ell_1 \dots \ell_n)$, and thus the two products on the left and right-hand sides of the tensor sign correspond to the two cycles of $\sigma \take \tau$.
	
	Since we performed summation over $i_1$ and $i_2$, each transposition $\tau$ arises in two ways. The inner product $\ip{u(i_1),u(i_2)}$ is the same as $\ip{u(i_2),u(i_1)}$ since the inner-product space is real.  Furthermore, it will be clear from the computations below that switching the labels of $i_1$ and $i_2$ would not change the final contribution to our formula for $\eta_f'$, as the relabeling can be compensated by cyclic permutation of monomials under the trace.  Hence, each transposition is effectively counted twice, but this is compensated by the factor of $1/2$ in front of $L_V$.
	
	Next, when we apply $\partial_V \otimes \partial_V$, that amounts to choosing one term on the left and one term on the right to replace with a tensor sign.  Let us call $j_1 = k_p$ the index on the left that we replace with the tensor sign, and call $j_2 = \ell_q$ the index on the right that we replace with the tensor sign.  Thus, we can write
	\[
	\sigma = (i_1 k_1 \dots k_{p-1} j_1 k_{p+1} \dots k_m i_2 \ell_1 \dots \ell_q j_2 \ell_{q+1} \dots \ell_n).
	\]
	Note that the indices $i_1$, $i_2$, $j_1$, $j_2$ such that (after cyclic permutation) $i_1 < j_1 < i_2 < j_2$ correspond precisely to $\sigma$-crossings.  For such $i_1$, $i_2$, $j_1$, $j_2$, the corresponding term in $(\partial_V \otimes \partial_V) L_V f$ is
	\begin{align*}
		&\ip{u(i_1),u(i_2)} (u(j_1) \otimes u(j_2)) \\
		\otimes & Z_{i_1} x_{u(k_1)} Z_{k_1} \dots x_{u(k_{p-1})} Z_{k_{p-1}} \\
		\otimes & Z_{j_1} x_{u(k_{p+1})} Z_{k_{p+1}}  \dots x_{u(k_m)} Z_{k_m} \\
		\otimes & Z_{i_2} x_{u(\ell_1)} Z_{\ell_1} \dots x_{u(\ell_n)}  Z_{\ell_{q-1}} \\
		\otimes & Z_{j_1} x_{u(\ell_{q+1})} Z_{\ell_1} \dots x_{u(\ell_n)}  Z_{\ell_n}.
	\end{align*}
	Finally, by application of $\mathfrak{n}_V$ tensor the maps $(\id_{V \oplus W} \otimes e_j)_*$ for $j = 1$, \dots, $4$, and the multiplication map $\mathfrak{m}_{4,3,2,1}$, we get
	\begin{align}
		h_{\tau,\tau'}(x_{V\otimes e_1},&x_{W\otimes e_1}, x_{V \otimes e_{2}},x_{W \otimes e_{2}},x_{V \otimes e_3}, x_{W \otimes e_3}, x_{V \otimes e_4}, x_{W \otimes e_4}) \label{eq:Dcrossterm} \\
		:=&\ip{u(i_1),u(i_2)}_V \ip{u(j_1),u(j_2)}_{V \oplus W}\nonumber \\
		&Z_{j_1} x_{u(\ell_{q+1}) \otimes e_4} Z_{\ell_1} \dots x_{u(\ell_n) \otimes e_4}  Z_{\ell_n} \nonumber \\
		&Z_{i_2} x_{u(\ell_1) \otimes e_3} Z_{\ell_1} \dots x_{u(\ell_n) \otimes e_3}  Z_{\ell_{q-1}} \nonumber \\
		&Z_{j_1} x_{u(k_{p+1}) \otimes e_2} Z_{k_{p+1}}  \dots x_{u(k_m) \otimes e_2} Z_{k_m} \nonumber \\
		&Z_{i_1} x_{u(k_1) \otimes e_1} Z_{k_1} \dots x_{u(k_{p-1}) \otimes e_1} Z_{k_{p-1}}. \nonumber
	\end{align}
	Here $\ip{u(j_1),u(j_2)}_{V \oplus W} = \ip{u(j_1),u(j_2)}_V$ because we assumed $u(j_1), u(j_2) \in V$, but below in handling $D_{V,W}^{\crss}$ we will also use \eqref{eq:Dcrossterm} in the case $u(j_1), u(j_2) \in W$.  Note that the order of multiplication follows the cycle
	\[
	(\ell_{q+1} \dots \ell_n \ell_1 \dots \ell_{q-1} k_{p+1} \dots k_m k_1 \dots k_{p-1}) = (\sigma \take \tau) \take \tau' = \sigma \take \tau \tau',
	\]
	where $\tau'$ is the transposition $(j_1 j_2)$.  Now $D_{V,V}^{\crss} f$ is the summation of this expression over appropriate $i_1$, $i_2$, $j_1$, and $j_2$, or equivalently over pairs $\tau$, $\tau'$ of transpositions that form a $\sigma$-crossing, with $\tau$ and $\tau'$ supported in $[r] \setminus S$.
	
	$D_{V,W}^{\crss}f$ is obtained in the analogous way, where we use $\partial_W \otimes \partial_W$ rather than $\partial_V \otimes \partial_V$, apply $\mathfrak{n}_W$ rather than $\mathfrak{n}_V$, and take $\tau'$ supported in $S$ rather than $[r] \setminus S'$, while $\tau$ is still supported in $S$.  Here in \eqref{eq:Dcrossterm}, we have $\ip{u(j_1),u(j_2)}_{V\oplus W} = \ip{u(j_1),u(j_2)}_W$.  Overall,
	\[
	(1-t) D_{V,V}^{\crss} f + D_{V,W}^{\crss} f = \sum_{\substack{\tau, \tau' \in \Trans([r]) \\ \supp(\tau), \supp(\tau') \subseteq [r] \setminus S}} (1 - t) h_{\tau,\tau'} + \sum_{\substack{\tau, \tau' \in \Trans([r]) \\ \supp(\tau) \subseteq [r] \setminus S \\ \supp(\tau') \subseteq S}} h_{\tau,\tau'}.
	\]
	
	The next operation in the definition of $T_{V,W}^{\crss}$ is to apply the map $(\alpha_{s,t})_*$, where $\alpha_{s,t}: V \oplus W \to V \oplus [(V \oplus W) \otimes \R^6]$ is given by \eqref{eq:alphast}.  The map $\alpha_{s,t}$ will replace $x_{v \otimes e_j}$ by $x_{\alpha_{s,t}(v \otimes e_j)}$ for $v \in V$ and $j \in [4]$.  More explicitly, evaluating $(\alpha_{s,t})_* D_{V,V}^{\crss} f$ on $X_V^{(N)}$ and $X_{(V \oplus W) \otimes \R^6}$ produces
	\begin{align*}  
		(\alpha_{s,t})_* &D_{V,V}^{\crss} f(X_V^{(N)},X_{(V \oplus W) \otimes \R^6}) \\
		= \Bigl[D_{V,V}^{\crss}f \bigl(&t^{1/2} X_V^{(N)} + (1 - t)^{1/2}(s^{1/2} X_{V \otimes e_5} + (1 - s)^{1/2} X_{V \otimes e_1}), s^{1/2} X_{W \otimes e_5} + (1 - s)^{1/2} X_{W \otimes e_1}, \\
		&t^{1/2} X_V^{(N)} + (1 - t)^{1/2}(s^{1/2} X_{V \otimes e_6} + (1 - s)^{1/2} X_{V \otimes e_2}), s^{1/2} X_{W \otimes e_6} + (1 - s)^{1/2} X_{W \otimes e_2}, \\
		&t^{1/2} X_V^{(N)} + (1 - t)^{1/2}(s^{1/2} X_{V \otimes e_6} + (1 - s)^{1/2} X_{V \otimes e_2}), s^{1/2} X_{W \otimes e_6} + (1 - s)^{1/2} X_{W \otimes e_2}, \\
		&t^{1/2} X_V^{(N)} + (1 - t)^{1/2}(s^{1/2} X_{V \otimes e_5} + (1 - s)^{1/2} X_{V \otimes e_4}), s^{1/2} X_{W \otimes e_5} + (1 - s)^{1/2} X_{W \otimes e_4}
		\bigr)\Bigr],
	\end{align*}
	and similarly with $D_{V,V}^{\crss}f$ replaced by $D_{V,W}^{\crss} f$.
	
	\begin{lemma}
		Fix $\tau = (i_1i_2)$ and $\tau = (j_1j_2)$ disjoint transpositions of $[r]$ with $\supp(\tau) \subseteq [r] \setminus S$.  Let $h_{\tau,\tau'}$ be given by \eqref{eq:Dcrossterm}.   For $\pi \in \Perm_2([r])$, let
		\[
		\crss(\pi,\tau) = \# \{\rho \in \Cyc(\pi): \rho \text{ and } \tau \text{ are a } \sigma\text{-cross}\}.
		\]
		Then when $u(j_1), u(j_2) \in V$, we have
		\begin{align*}
			(1-t)^{\mathbf{1}_{\supp(\tau') \subseteq [r] \setminus S}} &\frac{1}{N^2} \E \tr_{\M_N * \cM_{(V \oplus W) \otimes \R^6}} \\
			\Bigl[h_{\tau,\tau'}\bigl(&t^{1/2} X_V^{(N)} + (1 - t)^{1/2}(s^{1/2} X_{V \otimes e_5} + (1 - s)^{1/2} X_{V \otimes e_1}), s^{1/2} X_{W \otimes e_5} + (1 - s)^{1/2} X_{W \otimes e_1}, \\
			&t^{1/2} X_V^{(N)} + (1 - t)^{1/2}(s^{1/2} X_{V \otimes e_6} + (1 - s)^{1/2} X_{V \otimes e_2}), s^{1/2} X_{W \otimes e_6} + (1 - s)^{1/2} X_{W \otimes e_2}, \\
			&t^{1/2} X_V^{(N)} + (1 - t)^{1/2}(s^{1/2} X_{V \otimes e_6} + (1 - s)^{1/2} X_{V \otimes e_2}), s^{1/2} X_{W \otimes e_6} + (1 - s)^{1/2} X_{W \otimes e_2}, \\
			&t^{1/2} X_V^{(N)} + (1 - t)^{1/2}(s^{1/2} X_{V \otimes e_5} + (1 - s)^{1/2} X_{V \otimes e_4}), s^{1/2} X_{W \otimes e_5} + (1 - s)^{1/2} X_{W \otimes e_4}
			\bigr)\Bigr]
		\end{align*}
		\begin{equation} \label{eq:integralcombinatorics1}
			= \sum_{\substack{\pi \in \Perm_2([r]) \\ \tau, \tau' \in \Cyc(\pi)}} \sum_{\substack{A, B \subseteq [r] \\ A \sqcup B \sqcup S = [r] \\ \supp(\tau) \subseteq A \\ \supp(\tau') \subseteq B \cup S \\ \pi|_{S \cup B} \text{ $\sigma$-noncrossing}}} \frac{1}{N^{2g(\pi,\sigma)}} a(\pi) b(\pi) s^{\crss(\pi|_{S \cup B},\tau)-1} t^{\# \Cyc(\pi|_A)} (1 - t)^{\# \Cyc(\pi|_B)}.
		\end{equation}
	\end{lemma}
	
	\begin{proof}
		First, note that we can assume without loss of generality that $\supp(\tau')$ is either contained in $S$ or contained in $[r] \setminus S$.  Indeed, if $\supp(\tau')$ had one element $j_1$ in $S$ and one element $j_2$ in $[r] \setminus S$, then $\ip{u(j_1),u(j_2)}_{V \oplus W}$ would equal zero.  Hence, $h_{\tau,\tau'}$ would be zero and $a(\pi)$ would also be zero for every permutation $\pi$ having $\tau'$ as a cycle, so that both sides of \eqref{eq:integralcombinatorics1} would be zero.
		
		For $j \in [4]$, we denote by $I_j$ the set of indices $i$ such that $i$th vector is in $(V \oplus W) \otimes e_j$.  Thus, the cycle $\sigma$ can be written in cyclic order as
		\[
		i_1 \quad I_1 \quad j_1 \quad I_2 \quad i_2 \quad I_3 \quad j_2 \quad I_4.
		\]	
		We expand the left-hand side as a sum of the expected traces of monomials using multilinearity.  For each index $i \not \in S$, we make a choice of either $t^{1/2} X_{u(i)}^{(N)}$ or $(1 - t)^{1/2} X_{\alpha_{s,t}(u(i) \otimes e_j)}$.  Let $S_0 = S \setminus \supp(\tau')$.  Let $A_0$ be the set of indices where we choose $t^{1/2} X_{u(i)}^{(N)}$ and let $B_0$ be the set where we choose $(1 - t)^{1/2} X_{\alpha_{s,t}(u(i) \otimes e_j)}$.  Thus, $A_0$, $B_0$, and $S_0$ form a partition of $[r] \setminus \supp(\tau \tau')$.
		
		Let $\tilde{Z}_j$ for $j \in S \setminus \supp(\tau \tau')$ be the scalar matrices obtained from the $Z_k$'s by the following modifications in the case where $I_1$, \dots, $I_4$ are all nonempty:
		\begin{itemize}
			\item Replace $Z_{k_{p-1}}$ by $Z_{k_{p-1}} Z_{j_1}$.
			\item Replace $Z_{\ell_n}$ by $Z_{\ell_n} Z_{i_2}$.
			\item Replace $Z_{\ell_{q-1}}$ by $Z_{\ell_{q-1}} Z_{j_1}$.
			\item Replace $Z_{k_m}$ by $Z_{k_m} Z_{i_1}$.
		\end{itemize}
		In the case where one of the intervals $I_j$ is empty, we combine the terms even further.  For instance, if $I_1$ is empty and $I_2$ and $I_4$ is nonempty, then we replace $Z_{k_m}$ by $Z_{k_m} Z_{i_1} Z_{j_1}$.  In short, we combine terms wherever multiple $Z_k$'s appear in a row cyclically in \eqref{eq:Dcrossterm}.  If all the intervals $I_j$ are empty, then there are no remaining $x_v$'s and we can evaluate the trace $\tr_N(Z_{i_1} Z_{j_1} Z_{i_2} Z_{j_2})$ directly; we leave it to the reader to verify the lemma in this case.  Note that the terms $\tilde{Z}$ described above are precisely $\bZ^{\sigma,\tau \tau'}$ as given by Notation \ref{not:updateZ}, and when all the intervals are empty, then $\supp(\tau \tau') = [r] = [4]$, and $\tr(Z_{i_1}Z_{j_1} Z_{i_2} Z_{j_2}) = \lambda_{\sigma,\tau\tau'}$.
		
		Now we apply the genus expansion for mixtures of GUE and semicirculars (Proposition \ref{prop:genusexpansion}) to evaluate the expected trace of $h_{\tau,\tau'}$ as the scalar term $\ip{u(i_1),u(i_2)} \ip{u(j_1),u(j_2)}$ times
		\begin{equation} \label{eq:integralcombinatorics2}
			\sum_{\substack{A_0, B_0 \subseteq [r] \setminus \supp(\tau\tau') \\ A_0 \sqcup B_0 \sqcup S_0 = [r] \setminus \supp(\tau \tau')}} \sum_{\substack{\pi_0 \in \Perm_2([r] \setminus \supp(\tau \tau')) \\ \pi_0|_{B_0 \cup S_0} \text{ $\sigma \take \tau \tau'$-noncrossing}}} \frac{1}{N^{2g(\sigma \take \tau \tau',\pi)}} a_{s,t}(\pi_0) b_{\sigma \take \tau}(\pi_0),
		\end{equation}
		where
		\[
		a_{s,t}(\pi) = \prod_{(ij) \in \Cyc(\pi_0)} \ip{\tilde{u}_{A_0,B_0}(i), \tilde{u}_{A_0,B_0}(j)},
		\]
		where for each $i$,
		\[
		\tilde{u}_{A_0,B_0}(i) = \begin{cases}
			t^{1/2} u(i), & i \in A_0 \\
			(1 - t)^{1/2} u(t) \otimes (s^{1/2} e_{5 + \delta(j,2) + \delta(j,3)} + (1 - s)^{1/2} e_j) & i \in B_0 \cap I_j \\
			u(t) \otimes (s^{1/2} e_{5 + \delta(j,2) + \delta(j,3)} + (1 - s)^{1/2} e_j), & i \in S_0 \cap I_j,
		\end{cases}
		\]
		and where
		\[
		b_{\sigma \take \tau}(\pi_0) = \prod_{\gamma \in \Cyc(\pi_0 (\sigma \take \tau \tau'))} \tr_N\left( \prod_{k \in \gamma} \tilde{Z}_k \right).
		\]
		For each $\pi_0$, let $\pi \in \Perm_2([r])$ be the permutation with $\Cyc(\pi) = \Cyc(\pi_0) \sqcup \{\tau,\tau'\}$.  Let
		\[
		A = A_0 \cup \supp(\tau).
		\]
		Moreover, let $B$ be defined by adding $\supp(\tau')$ to $B_0$ if $u(j_1), u(j_2) \in V$ and $B = B_0$ otherwise.  Let $S$ be defined by adding $\supp(\tau')$ to $S_0$ if $u(j_1), u(j_2) \in W$ and $S = S_0$ otherwise.
		
		Our goal is to write each of the components of the formula above in terms of $\pi$, $A$, and $B$.
		
		\textbf{Scalar matrix term:} By Lemma \ref{lem:updateZ}, we have $b_{\sigma \take \tau \tau'}(\pi_0) = b(\pi)$ as given in Lemma \ref{lem:etaformula}.
		
		\textbf{Inner product term:} Note that for $i, i' \in S_0$,
		\begin{align*}
			\ip{\tilde{u}_{A_0,B_0}(i), \tilde{u}_{A_0,B_0}(i')} = \begin{cases}
				\ip{u(i), u(i')}, & (i,i') \in I_1^{\times 2} \cup I_2^{\times 2} \cup I_3^{\times 2} \cup I_4^{\times 2}, \\
				s \ip{u(i),u(i')}, & (i,i') \in [(I_1 \cup I_2) \times (I_3 \cup I_4)] \cup [(I_3 \cup I_4) \times (I_1 \cup I_2)] \\
				0, & (i,i') \in [(I_1 \cup I_4) \times (I_2 \cup I_3)] \cup [(I_1 \cup I_4) \times (I_2 \cup I_3)]
			\end{cases}
		\end{align*}
		In other words, $\ip{\tilde{u}_{A_0,B_0}(i), \tilde{u}_{A_0,B_0}(i')}$ is the same as $\ip{u(i),u(i')}$ if the two indices were in the same $I_j$, it gets multiplied by $s$ if $(i,i')$ crosses $\tau$, and it vanishes if $(i,i')$ crosses $\tau'$.  The case where $i, i' \in B_0$ has the same formula as the $S_0$ case above, except that all the terms are multiplied by $1 - t$.  Finally, in the case of $i, i' \in A_0$, we have $\ip{\tilde{u}_{A_0,B_0}(i),\tilde{u}_{A_0,B_0}(i)} = t \ip{u(i),u(i')}$.  Putting this all together,
		\[
		\ip{u(i_1),u(i_2)} \ip{u(j_1),u(j_2)} \prod_{(i,i') \in \Cyc(\pi_0)} \ip{\tilde{u}_{A_0,B_0}(i), \tilde{u}_{A_0,B_0}(j)}
		=
		a(\pi) s^{\crss(\pi|_{B_0 \cup S_0},\tau)} t^{|A_0|/2} (1 - t)^{|B_0|/2},
		\]
		if $\pi_0$ does not cross $\tau'$, and it is zero otherwise.
		
		\textbf{Noncrossing conditions:} We claim that, for the terms which do not vanish, $\pi|_{B \cup S}$ is $\sigma$-noncrossing if and only if $\pi|_{B_0 \cup S_0}$ is $\sigma \take (\tau \tau')$-noncrossing.  Indeed, suppose that $\pi|_{B_0 \cup S_0}$ is $\sigma \take (\tau \tau')$-noncrossing.  Then of course the restrictions to $I_1$, $I_2$, $I_3$, and $I_4$ are noncrossing.  We just showed that the inner product term vanishes if some $(i,i') \in  \Cyc(\pi_0)$ crosses $\tau'$, so we may assume there are no crossings of $\tau'$.  Hence, the only way that a crossing could occur is between $I_2$ and $I_3$ or $I_1$ and $I_4$.  However, this is also impossible because $I_1 \cup I_4$ occurs in a cyclically equivalent order with respect to both $\sigma$ and $\sigma \take (\tau \tau')$, and the same holds for $I_2 \cup I_3$.  One argues similarly that if $\pi|_{B \cup S}$ is $\sigma$-noncrossing and the term does not vanish, then $\pi|_{B_0 \cup S_0}$ is $\sigma \take \tau \tau'$-noncrossing.
		
		\textbf{Number of crossings of $\pi$ and $\tau$:} Note that $\pi|_{S \cup B}$ has one more crossing of $\tau$ than $\pi_0$ does, namely the crossing $\tau'$.  Hence, $\crss(\pi|_{B_0 \cup S_0},\tau) = \crss(\pi|_{B \cup S}, \tau) - 1$.
		
		\textbf{Powers of $t$ and $1 - t$:} Note $|A_0|/2 =|A|/2 - 1$.  In the case where $\supp(\tau') \subseteq B$, then $|B_0|/2 = |B|/2 - 1$, and the extra factor of $(1 - t)$ on the left-hand side of the equation compensates for this.  Otherwise, $|B_0|/2 = |B|/2$ and there is no extra factor of $1 - t$.
		
		\textbf{Genus term:} Since $\sigma$ is a single cycle, $\tau$ is not a $\sigma$-bridge.  This implies that, letting $\pi_1 = \pi|_{[r] \setminus \supp(\tau)}$, we have $g(\sigma \take \tau, \pi_1) = g(\sigma,\pi)$ by Lemma \ref{lem:genusupdate}. However, $\tau'$ is a $\sigma \take \tau$-bridge because $(\tau,\tau')$ is a $\sigma$-crossing.  Thus, $g(\sigma \take \tau \tau', \pi_0) = g(\sigma \take \tau,\pi_1) - 1$.  Overall,
		\[
		g(\sigma,\pi) = g(\sigma \take \tau \tau', \pi_0) + 1.
		\]
		Thus, we can replace $N^{-2g(\sigma \take \tau \tau',\pi_0)}$ with $N^2 N^{-2g(\sigma,\pi)}$.  This factor of $N^2$ is then moved to the left side of \eqref{eq:integralcombinatorics1} by multiplying through by $N^{-2}$.
		
		\textbf{Conclusion:} Substituting all these considerations into \eqref{eq:integralcombinatorics2}, we obtain \eqref{eq:integralcombinatorics1}.
	\end{proof}
	
	\begin{lemma} \label{lem:integralcomputation2}
		Fix $\tau \in \Trans([r])$ with support contained in $[r] \setminus S$.  Let $\mathbf{X}_{s,t}^{(N)}$ be the argument of $h_{\tau,\tau'}$ as in the previous lemma.  Then
		\begin{multline} \label{eq:integralcombinatorics3}
			\sum_{\substack{\tau' \in \Trans([r]) \\ \{\tau,\tau'\} \text{ $\sigma$-cross}}} \frac{1}{N^2} (1 - t)^{\mathbf{1}_{\supp(\tau') \subseteq [r] \setminus S}} \int_0^1 \E[h_{\tau,\tau'}(\bX_{s,t}^{(N)})]\,ds \\
			= \sum_{\substack{A, B \subseteq [r] \\ A \sqcup B \sqcup S = [r] \\ \supp(\tau) \subseteq A}} \sum_{\substack{\pi \in \Perm_2([r]) \\ \pi|_{B \cup S} \text{ $\sigma$-noncrossing} \\ \pi|_{B \cup S \cup \supp(\tau)} \text{ $\sigma$-crossing}}}
			\frac{1}{N^{2g(\sigma,\pi)}} a(\pi) b(\pi) t^{|A|/2-1} (1 - t)^{|B|/2}.
		\end{multline}
	\end{lemma}
	
	\begin{proof}
		We fix $\tau$, and then sum up \eqref{eq:integralcombinatorics1} over $\tau'$, $A$, $B$, and $\pi$.  For each choice of $\pi$, $A$, and $B$, the number of possible choices of $\tau'$ satisfying the conditions in \eqref{eq:integralcombinatorics1} is the number of crossings of $\pi|_{B \cup S}$ with $\tau$, or $\crss(\pi|_{B \cup S},\tau)$.  Hence, we obtain
		\begin{multline*}
			\frac{1}{N^2} \E \tr_{\M_N * \cM_{(V \oplus W) \otimes \R^6}}
			[h_{\tau,\tau'}(\bX_{s,t})] \\
			= \sum_{\substack{A, B \subseteq [r] \\ A \sqcup B \sqcup S = [r] \\ \supp(\tau) \subseteq A}} \sum_{\substack{\pi \in \Perm_2([r]) \\ \tau \in \Cyc(\pi)  \\ \pi|_{S \cup B} \text{ $\sigma$-noncrossing}}}
			\frac{1}{N^{2g(\pi,\sigma)}} a(\pi) b(\pi) \crss(\pi|_{B\cup S},\tau) s^{\crss(\pi|_{B \cup S},\tau)-1} t^{\# \Cyc(\pi|_A)} (1 - t)^{\# \Cyc(\pi|_B)}
		\end{multline*}
		Observe that
		\[
		\int_0^1 \crss(\pi|_{B\cup S},\tau) s^{\crss(\pi|_{B\cup S},\tau)-1}\,ds = \mathbf{1}_{\pi|_{B \cup S \cup \supp(\tau) \text{ $\sigma$-crossing}}},
		\]
		and substituting this into the above equation completes the proof.
	\end{proof}
	
	\begin{proof}[{Conclusion of the proof of Theorem \ref{thm:Parraud1}}]
		Let $f$ be as in Observation \ref{obs:specialf}.  The right-hand side of \eqref{eq:maingoal} is
		\[
		\frac{1}{N^2} \int_0^1 \mathbb{E} \circ \tr_{\M_N * \cM_{(V \oplus W) \otimes \R^6}} \left[(\alpha_{s,t})_*[(1-t)D_{V,V}^{\crss} + D_{V,W}^{\crss}](f)(X_V^{(N)},X_{(V \oplus W) \otimes \R^6}) \right] \,ds;
		\]
		this can be evaluated by summing up the quantity in Lemma \ref{lem:integralcomputation2} over transpositions $\tau \in \Trans([r])$ supported in $[r] \setminus S$.  This produces the same result as $\eta_f'(t)$ as computed in Lemma \ref{lem:etaderivative}.  It follows that \eqref{eq:maingoal} holds, and therefore by Observation \ref{obs:specialf} and Observation \ref{obs:derivativeintegral}, the proof of Theorem \ref{thm:Parraud1} is complete.
	\end{proof}

	\bibliographystyle{plain}
	\bibliography{asymptotic_expansion_bib}

\begin{thebibliography}{10}

\bibitem{AGZ2009}
Greg~W. Anderson, Alice Guionnet, and Ofer Zeitouni.
\newblock {\em An Introduction to Random Matrices}.
\newblock Cambridge Studies in Advanced Mathematics. Cambridge University
  Press, 2009.

\bibitem{BelCap2022}
Serban Belinschi and Mireille Capitaine.
\newblock Strong convergence of tensor products of independent {G.U.E.}
  matrices.
\newblock preprint, arXiv:2205.07695, 2022.

\bibitem{BorCol2023}
Charles Bordenave and Benoit Collins.
\newblock Norm of matrix-valued polynomials in random unitaries and
  permutations, 2023.
\newblock preprint, arXiv:2304.05714.

\bibitem{BIPZ1978}
E.~Br{\'e}zin, C.~Itzykson, G.~Parisi, and J.~B. Zuber.
\newblock Planar diagrams.
\newblock {\em Commun. math. Phys.}, 59:35--51, 1978.

\bibitem{Cebron2013}
Guillaume C{\'e}bron.
\newblock Free convolution operators and free hall transform.
\newblock {\em Journal of Functional Analysis}, 265(11):2645 -- 2708, 2013.

\bibitem{CGP2022}
Beno{\^i}t Collins, Alice Guionnet, and F{\'e}lix Parraud.
\newblock On the operator norm of non-commutative polynomials in deterministic
  matrices and iid {GUE} matrices.
\newblock {\em Cambridge Journal of Mathematics}, 10(1):195--260, 2022.

\bibitem{DGS2021}
Yoann Dabrowski, Alice Guionnet, and Dimitri Shlyakhtenko.
\newblock Free transport for convex potentials.
\newblock {\em New Zealand J. Math.}, 52:259--359, 2021.

\bibitem{DHK2013}
Bruce~K. Driver, Brian~C. Hall, and Todd Kemp.
\newblock The large-$n$ limit of the {S}egal-{B}argmann transform on $u_n$.
\newblock {\em Journal of Functional Analysis}, 265(11):2585 -- 2644, 2013.

\bibitem{EM2003partition}
N.~M. Ercolani and K.~R. McLaughlin.
\newblock Asymptotics of the partition function for random matrices via
  {R}iemann-{H}ilbert techniques and applications to graphical enumeration.
\newblock {\em International Mathematics Research Notices}, 2003:755--820,
  2003.

\bibitem{CGVVH2024}
Chi fang Chen, Jorge Garza-Vargas, and Ramon {Van Handel}.
\newblock A new approach to strong convergence {II}. the classical ensembles.
\newblock Preprint, arXiv:2412.00593, 2024.

\bibitem{PS2023energy}
Kevin~Schnelli F{\'e}lix~Parraud.
\newblock The free energy of matrix models.
\newblock Preprint, arXiv:2310.12948, 2023.

\bibitem{GuionnetParkCity}
Alice Guionnet.
\newblock {\em Statistical Mechanics and Random Matrices}.
\newblock IAS/Park City Mathematics Series. American Mathematical Society and
  Institute for Advanced Study, 2009.

\bibitem{GMS2006}
Alice Guionnet and Edouard Maurel-Segala.
\newblock Combinatorial aspects of random matrix models.
\newblock {\em Latin American Journal of Probability and Statistics (ALEA)},
  1:241--279, 2006.

\bibitem{HaagThorbNormBound}
Uffe Haagerup and Steen Thorbj{{\o}}rnsen.
\newblock Random matrices with complex {G}aussian entries.
\newblock {\em Expo. Math.}, 21(4):293--337, 2003.

\bibitem{HZ1986}
J.~Harer and D.~Zagier.
\newblock The {E}uler characteristic of the moduli space of curves.
\newblock {\em Inventiones mathematicae}, 85:457--485, 1986.

\bibitem{HayesPT}
Ben Hayes.
\newblock A random matrix approach to the peterson-thom conjecture.
\newblock {\em Indiana Univ. Math. J.}, 71(3):1243--1297, 2022.

\bibitem{JekelThesis}
David Jekel.
\newblock {\em Evolution equations in non-commutative probability}.
\newblock PhD thesis, University of California, Los Angeles, 2020.

\bibitem{JLS2022}
David Jekel, Wuchen Li, and Dimitri Shlyakhtenko.
\newblock Tracial non-commutative smooth functions and the free {W}asserstein
  manifold.
\newblock {\em Dissertationes Mathematicae}, 580:1--150, 2022.

\bibitem{MingoSpeicher}
James~A. Mingo and Roland Speicher.
\newblock {\em Free Probability and Random Matrices}.
\newblock Fields Institute Monographs. Springer, New York, NY, 2017.

\bibitem{Nikitopoulos2023}
Evangelos~A. Nikitopoulos.
\newblock Noncommutative $c^k$ functions and {F}r{\'e}chet derivatives of
  operator functions.
\newblock {\em Expositiones Mathematicae}, 41:115--163, 2023.

\bibitem{Parraud2023unitary}
F{\'e}lix Parraud.
\newblock Asymptotic expansion of smooth functions in deterministic and iid
  {H}aar unitary matrices, and application to tensor products of matrices.
\newblock Preprint, arXiv:2302.02943, 2023.

\bibitem{Parraud2023}
F{\'e}lix Parraud.
\newblock Asymptotic expansion of smooth functions in polynomials in
  deterministic matrices and iid {GUE} matrices.
\newblock {\em Communications in Mathematical Physics}, 399(1):249--294, 2023.

\bibitem{Parraud2024tensor}
F{\'e}lix Parraud.
\newblock The spectrum of a tensor of random and deterministic matrices.
\newblock Preprint, arXiv:2410.04481, 2024.

\bibitem{PetersonThom}
Jesse Peterson and Andreas Thom.
\newblock Group cocycles and the ring of affiliated operators.
\newblock {\em Invent. Math.}, 185(3):561--592, 2011.

\bibitem{Rains1997}
E.~M. Rains.
\newblock Combinatorial properties of {B}rownian motion on the compact
  classical groups.
\newblock {\em Journal of Theoretical Probability}, 10(3):659--679, 1997.

\bibitem{Schultz2005}
Hanne Schultz.
\newblock Non-commutative polynomials of independent gaussian random matrices.
  the real and symplectic cases.
\newblock {\em Probab. Theory Related Fields}, 131:261--309, 2005.

\bibitem{Shcherbina2014}
Mariya Shcherbina.
\newblock {\em Asymptotic expansions for $\beta$-matrix models and their
  applications to the universality conjecture}, volume~65 of {\em MSRI
  Publications}, pages 463--482.
\newblock 2014.

\bibitem{ShlyakhtenkoParkCity}
Dimitri Shlyakhtenko.
\newblock {\em Random matrices and free probability}, pages 390--460.
\newblock IAS/Park City Mathematics Series. American Mathematical Society and
  Institute for Advanced Study, 2019.

\bibitem{tHooft1974A}
G.~{'t Hooft}.
\newblock A planar diagram theory for strong interactions.
\newblock {\em Nuclear Physics B}, 72(3):461--473, 1974.

\bibitem{tHooft1974B}
G.~{'t Hooft}.
\newblock A two-dimensional model for mesons.
\newblock {\em Nuclear Physics B}, 75(3):461--470, 1974.

\bibitem{VoiculescuFE5}
Dan-Virgil Voiculescu.
\newblock The analogues of entropy and of {F}isher's information in free
  probability {V}.
\newblock {\em Inventiones Mathematicae}, 132:189--227, 1998.

\bibitem{voiculescu1992free-random}
Dan-Virgil Voiculescu, Kenneth~J. Dykema, and Alexandru Nica.
\newblock {\em Free Random Variables}, volume~1 of {\em CRM Monograph Series}.
\newblock American Mathematical Society, Providence, 1992.

\bibitem{Zvonkin1997}
A.~Zvonkin.
\newblock Matrix integrals and map enumeration: An accessible introduction.
\newblock {\em Mathematical and Computer Modelling}, 26(8):281--304, 1997.

\end{thebibliography}
	
\end{document}